\documentclass[12pt,twoside]{amsart}
\input{includeNice3}
\usepackage{mabliautoref}
\usepackage{amssymb}
\usepackage{amscd}
\usepackage[abbrev,alphabetic]{amsrefs}
\usepackage{hyperref, cleveref}
\usepackage[OT2,T1]{fontenc}
\DeclareSymbolFont{cyrletters}{OT2}{wncyr}{m}{n}
\DeclareMathSymbol{\Sha}{\mathalpha}{cyrletters}{"58}

\newcommand{\Bref}[1]{(\hyperref[item:Brauer]{\ensuremath{B_{#1}}})}
\newcommand{\PIDref}[1]{(\hyperref[item:PID]{\ensuremath{PID_{#1}}})}
\newcommand{\AFref}[1]{(\hyperref[item:AbrashkinFontaine]{\ensuremath{AV_{#1}}})}
\newcommand{\Hodgeref}[1]{(\hyperref[item:HodgeNumbers]{\ensuremath{h_{#1}^{p,q}}})}

\usepackage[a4paper,margin=1in]{geometry}  

\usepackage{xparse}
\let\realItem\item 
\makeatletter
\NewDocumentCommand\myItem{ o }{%
   \IfNoValueTF{#1}%
      {\realItem}
      {\realItem[#1]\def\@currentlabel{#1}}
}
\usepackage{enumitem}
\setlist[enumerate]{
    before=\let\item\myItem,       
    label=\textnormal{(\arabic*)}, 
    widest=(2')                    
}

\usepackage{xypic}
\usepackage{tikz-cd}
\usepackage{caption}

\usepackage{url}
\RequirePackage{booktabs, multirow}
\RequirePackage{pgf}

\title[]
{On surfaces with smooth projective models over $\mathbb{Z}$} 
\author{Fabio Bernasconi, Gebhard Martin and Zsolt Patakfalvi} 
\subjclass[2020]{}
\keywords{}

\newcommand{\FB}[1]{\textcolor{orange}{(FB: #1)}}

\address{Dipartimento di Matematica “Guido
Castelnuovo”, SAPIENZA Università di Roma, Piazzale Aldo Moro 5, I-00185,
Roma} 
\email{fabio.bernasconi@uniroma1.it}

\address{Mathematisches Institut  Universität Bonn, Endenicher Allee 60, 53115 Bonn, Germany}
\email{gmartin@math.uni-bonn.de} 

\address{EPFL SB MATH CAG
	MA C3 615 (B\^atiment MA)
	Station 8
	CH-1015 Lausanne}
\email{zsolt.patakfalvi@epfl.ch}

\newcommand{\Z}{\mathbb{Z}}

\begin{document}

\maketitle

\begin{abstract}
    In this expository article, we prove a birational classification of smooth projective models of surfaces with negative Kodaira dimension over $\mathbb{Z}$ and over more general rings of integers $\mathcal{O}_K$, depending on their arithmetic and cohomological invariants. 
    Along the way we collect some results on smooth projective models of surfaces over Dedekind domains.
\end{abstract}

\tableofcontents

\section{Introduction}

At the ICM held in Stockholm in 1962 \cite{Saf63}, Šafarevič asked whether there exist smooth projective curves of genus at least 1 defined over $\mathbb{Q}$ having good reduction at every prime $p$.
This was later answered in the negative by Abrashkin and Fontaine in \cite{Fon85, Abr90}, who showed in fact a stronger result: there are no smooth abelian schemes over $\mathbb{Z}$.
It is natural to extend Šafarevič’s question to higher dimensions:

\begin{center}
    $(\Sha):$ \emph{What are the smooth projective schemes over $\mathbb{Z}$?}
\end{center}

As reported by Mazur, Grothendieck \cite[pp. 242–243]{Maz86}, and Schröer \cite[Introduction]{Sch23}, the only examples that come to mind, up to birational transformations, are flag varieties, toric varieties, Hilbert schemes of points on smooth models of rational surfaces and projectivisations of vector bundles over such schemes. 
An interesting open question is whether there exist smooth projective schemes over $\mathbb{Z}$ with non-negative Kodaira dimension.

Question $(\Sha)$ has attracted notable research interest in recent years: a groundbreaking result of Schr\"{o}er \cite{Sch23} shows that there are no Enriques surfaces  over $\mathbb{Z}$, thus showing that smooth models of surfaces with trivial first Chern class over $\mathbb{Z}$ do not exist. 
Moreover, the existence and classification of certain classes of Fano 3-folds and $n$-folds (such as $V_5$, $V_{22}$ and Mukai $n$-folds of genus 7) over the integers have been recently studied in \cite{ito2024arithmeticfinitenessmukaivarieties, ito2025quinticdelpezzothreefolds, itov22}, giving the first examples which are not combinatorial in nature. 

Another possible generalisation of the original question of Šafarevič is to ask which smooth projective schemes exist when the base $\mathbb{Z}$ is replaced by a more general ring of integers in terms of their arithmetic and cohomological invariants. 
In this case, the case of families of abelian varieties over rings of integers of cyclotomic fields and quadratic fields have been thoroughly investigated by Schoof in \cite{Sch03, Sch25}, and for example Sch\"utt \cite{Schütt} and Schr\"{o}er \cite{schröer2025k3surfacessmallnumber} have constructed smooth models of K3 and Enriques surfaces over the ring of integers of $\mathbb{Q}(\sqrt{7})$, $\mathbb{Q}(\sqrt[3]{2}, \zeta_3)$, and of many other quadratic and cubic number fields.

In this note, we address Question $(\Sha)$ for smooth projective families of surfaces over the integers and, more generally, over Dedekind schemes, from both birational and biregular perspectives. 
First, in \autoref{section: general_facts} we recollect the known results for smooth models over rings of integers and Dedekind schemes in the literature which we will use throughout this note.
From the birational classification perspective, we then apply techniques from the Minimal Model Program (MMP) for arithmetic threefolds, developed in recent years by several authors \cite{7authors, TY23, BBS24, Sti24, HW23, XX24}. Since we work with smooth families of surfaces over Dedekind schemes, flips do not occur at any step of the MMP; this allows us to quickly reprove the main results of the MMP in \autoref{section:MMP}.

From the biregular perspective, we then bound the invariants of smooth surfaces arising as end products of the MMP (namely, Mori fibre spaces or minimal models), in terms of the arithmetic and cohomological properties of the base Dedekind scheme.
The case of Mori fibre spaces is treated in detail in  \autoref{sec: kod_-infty}, where we study smooth projective models of $\mathbb{P}^2$ and Hirzebruch surfaces over Dedekind domains.

Gathering all the previous results, in \autoref{sec: smooth_Z} we specialise to the study of smooth models of surfaces over $\mathbb{Z}$ and obtain the following restrictions on the possible cases:
\begin{theorem}[See \autoref{thm: rough_classification_Z}] \label{rough_class_Z}
    Let $X \to \Spec(\mathbb{Z})$ be a projective smooth morphism, where $X$ is integral and $\dim(X)=3$. Then, $X$ is the successive blow-up at $\mathbb{Z}$-points of one of the following smooth projective schemes $Y \to \Spec(\mathbb{Z})$:
    \begin{enumerate}
        \item \label{rough_class_Z:Hirzebruch} $Y \simeq \mathbb{P}^2_{\mathbb{Z}}$ or $Y$ is a Hirzebruch surface over $\bZ$. The latter means that $Y\simeq \mathbb{P}_{\mathbb{P}^1_{\mathbb{Z}}}(\mathcal{E})$ where $\mathcal{E}$ is a rank $2$ vector bundle corresponding to an extension class in $\Ext^1(\mathcal{O}_{\mathbb{P}^1_{\mathbb{Z}}}, \mathcal{O}_{\mathbb{P}^1_{\mathbb{Z}}}(n))$ for some integer $n$;
        \item a $K_Y$-trivial fibration $Y \to \mathbb{P}^1_\mathbb{Z}$ such that $Y_{\mathbb{Q}}$ has Kodaira dimension 1;
        \item the divisor $K_Y$ is big and nef with  $K_{Y_\mathbb{Q}}^2 \leq 9$.
    \end{enumerate}
\end{theorem}

We leave open the (very!) interesting question of whether examples of smooth models appearing in cases (2) and (3) of \autoref{rough_class_Z} actually exist.

As the classification of vector bundles of rank 2 on $\mathbb{P}^1_{\mathbb{Z}}$ is a hard problem (cf. \cite{Smirnov15,Polyakov}), we study instead the birational geometry of smooth models of Hirzebruch surfaces appearing in (1) of \autoref{thm: rough_classification_Z} over general Dedekind domains.
Recall that, over a field $k$, by Grothendieck's splitting principle every vector bundle $\mathcal{E}$ of rank $2$ over $\mathbb{P}^1_k$ satisfies $\mathcal{E} \simeq \mathcal{O}_{\mathbb{P}^1_k}(i) \oplus \mathcal{O}_{\mathbb{P}^1_k}(j)$ for some $j \geq i$ and the integer $j-i$ is called the \emph{degree} ${\rm deg}(Y)$ of the associated Hirzebruch surface $Y \coloneqq \mathbb{P}_{\mathbb{P}^1_k}(\mathcal{E})$.
The easiest example of a smooth model of a degree $n$ Hirzebruch surface over $\mathbb{Z}$ is thus $\bP_{\bP^1_{\bZ}} (\sO_{\bP^1_{\bZ}} \oplus \sO_{\bP^1_{\bZ}}(n))$, but there are many more.
However, in \autoref{prop:dec_elem_transf}, we show that any two models are connected via a sequence of birational operations induced by elementary transformations (see \autoref{def:elementary_transformation}), which yields the following: 

\begin{proposition}
\label{prop:Hirzebruch_surface_intro}
In point \ref{rough_class_Z:Hirzebruch} of \autoref{rough_class_Z},
\begin{enumerate}
\item \label{item:parity} for all $p$ prime, the difference $\deg(Y_p) - \deg(Y_{\mathbb{Q}})$ is an even non-negative integer;
\item \label{item:all_appear} all numerical possibilities for the degrees of $Y_p$ satisfying point \ref{item:parity}  of the present proposition do appear. 
\item \label{item:elementary_transf} any two Hirzebruch surfaces over $\mathbb{Z}$ can be connected by a sequence of blow-ups and blow-downs  associated to elementary transformations to the corresponding vector bundles (see \autoref{def:elementary_transformation}, \autoref{prop:dec_elem_transf} and \autoref{prop:elementary_transf_to_blow_up}).
\end{enumerate}
\end{proposition}

\begin{remark}
Unfortunately, in general, fixing the degrees of all the $Y_p$ does not determine the isomorphism class of $Y$.
In fact, not even the isomorphism types of the vector bundle over each completion $\mathbb{Z}_p$ determine $Y$ completely. For vector bundles with small jumps, this is studied in \cite{Smi16,Polyakov}. Explicitly, in the notation of \cite[Theorem 6]{Polyakov}, the two vector bundles $V_0(21,2)$ and $V_0(21,1)$ are isomorphic over all the $\mathbb{Z}_p$ but not over $\mathbb{Z}$.
\end{remark}

For explicit equations that give a full parameterization under extra assumptions, we refer to \autoref{ex:equations}. We conclude this note with the classification of smooth models of del Pezzo surfaces over $\mathbb{Z}$.\\

\noindent \textbf{Acknowledgements.} 
The authors would like to thank Christian Liedtke for useful and interesting discussion on smooth models of varieties over number fields, and Julia Schneider and Matthias Sch\"utt for useful comments on the first version. 
We would like to thank the referee for useful comments that helped to improve the article.
We are grateful to the Bernoulli Center for the excellent working conditions provided during the workshop \emph{p-adic and characteristic p Methods in Algebraic Geometry} in June 2025.
This survey was written for the proceedings for the Summer Research Institute in Algebraic Geometry at Fort Collins, and we thank the organisers for the invitation.

\section{Preliminaries and general results} \label{section: general_facts}

\subsection{Notation}

\begin{itemize}
\item We work over an excellent Dedekind scheme (i.e. an excellent, regular, Noetherian, integral scheme of dimension $1$) $S$ whose closed points have perfect residue fields, unless stated otherwise.  
We denote by $K$ the function field of $S$. If $S$ is affine, we denote the corresponding ring by $R$.
\item Unless stated otherwise, our schemes $X$ are separated and of finite type over $S$. 
\item For a point $s \in S$, we denote by $X_{s} \coloneqq X \times_S \Spec(k(s))$ the fibre over $s$. We let $X_K$ be the generic fiber of $X$.
\item If $S = \Spec (\mathbb{Z})$, we set $X_0 \coloneqq X_\mathbb{Q}$ and $X_p \coloneqq X_{(p)}$.
\item For $X$ smooth and proper over $s$, the $(p,q)$-th Hodge number of $X_{s}$ is denoted $h^{p,q}(X_{s}) \coloneqq h^q(X_{s},\Omega_{X_{s}/k(s)}^p)$.
\item For $X$ smooth and proper over $S$, we let $b_i(X_{s})$ be the $i$-th $\ell$-adic Betti number of $X_{s}$, where $\ell \neq {\rm char}(k(s))$, and $\rho(X_{s})$ the Picard number of $X_{s}$. 
\item If $X$ is a scheme and $\mathcal{E}$ is a vector bundle on $X$, we call $\mathbb{P}_X(\mathcal{E}) \coloneqq \Proj_X \Sym^{\bullet} \mathcal{E}$  the projective bundle associated to $\mathcal{E}$. Given a field $k$, when no confusion can arise, we write $\mathbb{F}_n= \mathbb{P}_{\mathbb{P}^1_k}(\mathcal{O}_{\mathbb{P}^1_k} \oplus \mathcal{O}_{\mathbb{P}^1_k}(-n))$ for the Hirzebruch surface of degree $n$ over $k$.
\item Given a proper scheme $f: X \to S$ and a line bundle $L$ on $X$, we denote the $S$-algebra of $L$ by $R(X,L)= \bigoplus_{n \geq 0} f_* (L^{\otimes n})$.
\item A contraction $f\colon X \to Y$ is a proper morphism such that $f^{\#}\colon \mathcal{O}_Y \to f_*\mathcal{O}_X$ is an isomorphism.
\item Given a semiample line bundle $L$ on $X$, we say that $f \colon X \to Y$ is the induced semiample contraction if it is a contraction and $L \sim_{\mathbb{Q}} f^*A$, where $A$ is an ample $\mathbb{Q}$-divisor on $Y$. 
\item Given a proper morphism $f \colon X \to Z$ of Noetherian schemes, a \emph{curve over $Z$} (the reference to $Z$ will be omitted whenever it is clear from the context) is a proper 1-dimensional integral subscheme of $X$ which is contracted to a point by $f$. 
We say two Cartier divisors $D_1, D_2$ are numerically equivalent over $Z$ $D_1 \equiv_Z D_2$ if $D_1 \cdot C= D_2 \cdot C$ for any curve over $Z$. If $Z$ is affine, we simply write $\equiv$.
\item Given a proper morphism $f \colon X \to Z$ of Noetherian schemes, we define the relative (numerical) Néron-Severi group $N^1(X/Z)$ as $(\Pic(X)/\equiv_Z) \otimes \mathbb{R}$, which is finite-dimensional (it follows from \cite[Lemma 7.6]{CJM22}), and its dimension $\rho(X/Z)$ is the relative Picard rank of $f$.
We denote by $N_1(X/Z)$ the dual vector space of $N^1(X/Z)$.
\item Given a proper morphism $f \colon X \to Z$ of Noetherian schemes, we  define the \emph{Mori cone} $\NE(X/Z)_{\mathbb{R}}$ as the closure in $N_1(X/Z)_{\mathbb{R}}$ of the cone generated by $\mathbb{R}$-linear positive combination of curves $C$ over $Z$.
\item If $X$ is a scheme, the canonical sheaf is defined as the lowest non-zero cohomology sheaf of the normalised dualising complex: $\omega_X \coloneqq \mathcal{H}^{-\dim X}(\omega_X^{\bullet})$. 
If $X$ is normal, then $\omega_X$ is a reflexive sheaf of rank 1, and we denote by $K_X$ a Weil divisor such that $\omega_X \cong \mathcal{O}_X(K_X)$. For the singularities of the MMP (such as terminal or klt), we refer to \cite{kk-singbook}.
\item A \emph{Mori fibre space} is a contraction $X \to Z$ over $S$ where $X$ has $\mathbb{Q}$-factorial terminal singularities, $\dim(Z)< \dim(X)$, $-K_X$ is ample over $Z$ and $\rho(X/Z)=1$.
A \emph{minimal model} is a $\mathbb{Q}$-factorial terminal normal projective scheme $X$ over $S$ such that $K_X$ is nef over $S$.
\end{itemize}

\subsection{Some good properties of the integers} \label{sec: integers}

In this section, we define a number of favourable properties that $S$ can have and then recall which of them are known to be satisfied in the case $S = \Spec(\mathbb{Z})$. Consider the following possible properties of $S$ (resp.~of $R$ if $S = \Spec(R)$ is affine):

\begin{enumerate}
\item[($M$)] The \'etale fundamental group $\pi_1(S)$ of $S$ is trivial. \label{item:Minkowski} 
\smallskip
\item[(\ensuremath{B_i})] The Brauer group ${\rm Br}(S)$ of $S$ has trivial $i$-torsion.
\label{item:Brauer}
\smallskip
\item[(\ensuremath{PID_i})] Every locally free sheaf of rank $i$ on $S$ is isomorphic to $\mathcal{O}_S^{\oplus i}$.
 \label{item:PID}
 \smallskip
\item[(\ensuremath{AV_i})] 
There exists no non-trivial Abelian variety of relative dimension $\leq i$ over $S$.
\label{item:AbrashkinFontaine} 
\smallskip
\item[(\ensuremath{h_i^{p,q}})] If $X$ is smooth and proper over $R$, then $h^{p,q}(X_K) = 0$ for $p \neq q$ and $p + q \leq i$.
\label{item:HodgeNumbers}
\smallskip
\end{enumerate}

\begin{remark} If $S$ is affine, then any \ref{item:PID} for $i>0$ is equivalent to $R$ being a PID, by the structure theorem for modules over PIDs. In particular, \ref{item:Minkowski} implies every \ref{item:PID} for rings of integers in number fields, because their class group is finite.
Most other implications between the first three properties are known to fail in general: 
\begin{enumerate}
\item By \cite[p.441]{YamamuraImaginary}, the ring of integers $R$ of the maximal abelian unramified extension $K_1 = K(\sqrt{(23 + 3 \sqrt{137}/2})$ of $K = \mathbb{Q}(\sqrt{-1507})$ satisfies \ref{item:PID} for all $i$ but does not satisfy \ref{item:Minkowski}. Since $K_1$ has no real places, $R$ also satisfies \ref{item:Brauer} for all $i$ by \autoref{cor: Brauer_group_integers} below. \footnote{Note that the degree of $K_1$ is $8$ and it would be interesting to find examples of smaller degree $d$. By \autoref{cor: Brauer_group_integers} and \cite[p.417]{YamamuraImaginary}, there are no examples in degree $d = 2$.}
\item By \autoref{cor: Minkowski} and \autoref{cor: Brauer_group_integers} below, the ring of integers of $\mathbb{Q}(\sqrt{5})$ satisfies \ref{item:Minkowski} (and  thus every \ref{item:PID}), but not \Bref{2}.
\item The curve $S= \mathbb{P}^1_{\mathbb{C}}$ satisfies all \ref{item:Brauer} and \ref{item:Minkowski}, but not \PIDref{1}.
\end{enumerate}
\end{remark}

\begin{remark} \label{rmk: AF_implies_mixed}
Note that for any $i \geq 1$, the properties \ref{item:AbrashkinFontaine}
and \ref{item:HodgeNumbers} force $S$ to be of mixed characteristic, for otherwise there are constant families over $S$.
Even relaxing non-trivial to non-isotrivial would not really fix this issue, since there are non-isotrivial families of Abelian varieties even over $\mathbb{P}^1_k$, where $k$ is an algebraically closed field of positive characteristic \cite{MR81, Ros-Schroer}.
\end{remark}

\begin{example}
One interesting example of a (non-affine) excellent Dedekind scheme $S$ with perfect residue fields and satisfying \ref{item:Minkowski}, \Bref{i}, and \PIDref{i} aside from certain rings of integers is the punctured spectrum $S = \Spec (\mathbb{C}[[x,y]]) \setminus \left\{(x,y) \right\}$:

Indeed, $S$ satisfies \ref{item:Minkowski} and \Bref{i} by purity (for the \'etale fundamental group \cite[Tag 0BMA]{stacks-project} and the Brauer group \cite[Theorem 3.7.1.(i)]{brauergroupbook}, respectively) and since $\Spec (\mathbb{C}[[x,y]])$ is \'etale simply connected with trivial Brauer group. Moreover, all locally free sheaves on $S$ are trivial, because they extend to reflexive and hence free sheaves on the local ring $\mathbb{C}[[x,y]]$.
\end{example}

Minkowski's theorem (see \cite[Theorem III.2.18]{Neu99}) on the non-existence of unramified extensions of $\mathbb{Q}$ can be rephrased as follows:

\begin{theorem}[Minkowski] \label{thm: minkowski}
The ring $R = \mathbb{Z}$ satisfies \ref{item:Minkowski}.
\end{theorem}



There is a well-known criterion for property \ref{item:Brauer} in the case where $K$ is a global field, which follows from the following result \cite[Theorem 3.6.1.ii]{brauergroupbook}:

\begin{theorem}[Grothendieck] \label{thm: BrauerGrothendieck}
There is an exact sequence
$$
0 \to \Br(S) \to \Br(K) \to \bigoplus_{x \textrm{ closed}} H^1\big(k(x), \mathbb{Q}/\mathbb{Z} \big) \to H^3_{\et}(S,\mathbb{G}_m) \to H^3_{\et}(K,\mathbb{G}_m) 
$$
\end{theorem}

\begin{corollary} \label{cor: Brauer_group_integers}
    If $R$ is the ring of integers in a number field $K$ and $r$ is the number of real Archimedean places of $K$, then:
    \begin{enumerate}
    \item $R$ satisfies \Bref{i} for all $i \neq 2$.
    \item $R$ satisfies \Bref{2} if and only if $r \leq 1$.
    \end{enumerate}
\end{corollary}
\begin{proof}
The short exact sequence of global class field theory
$$ 0 \to \Br(K) \to \bigoplus_v \Br(K_v) \xrightarrow{\text{inv}} \mathbb{Q}/\mathbb{Z} \to 0  ,$$
    where $v$ runs among all the places of $K$ together with \autoref{thm: BrauerGrothendieck} shows that 
    $$\Br(R) \simeq  \ker\big(\bigoplus_{v \mid \infty} \Br(K_v) \xrightarrow{\text{inv}} \mathbb{Q}/\mathbb{Z}\big).$$
    As for a real Archimedean place $v$ we have $\Br(K_v) \simeq \mathbb{Z}/2\mathbb{Z}$ and $\text{inv}$ is non-vanishing, it follows that $\Br(R)\simeq (\mathbb{Z}/2\mathbb{Z})^{\oplus r-1}$.
\end{proof}

The fact that $\mathbb{Z}$ satisfies \ref{item:AbrashkinFontaine} for all $i$ is due to Abrashkin and Fontaine (see \cite{Fon85}, and \cite{Abr90, Fontaine_schemas_Z}). Since Abelian varieties have non-zero first Betti number, we have $\Hodgeref{1} \Rightarrow \AFref{i}$ for all $i$, which is why we can state the results of Abrashkin and Fontaine as follows:

\begin{theorem}[Abrashkin, Fontaine] \label{thm: abrashkin}
The following hold:
\begin{enumerate}
\item The rings $R$ of integers of $\mathbb{Q}(\sqrt{-1}), \mathbb{Q}(\sqrt{-3})$, and $\mathbb{Q}(\sqrt{5})$ satisfy \Hodgeref{2}.
\item The ring $R = \mathbb{Z}$ satisfies \Hodgeref{3} (resp. \Hodgeref{4}, assuming the generalized Riemann hypothesis). 
\end{enumerate}
\end{theorem}

\begin{remark}
We do not claim that this list is complete. For example, very recently, Schoof \cite[Theorem 1.1]{Sch25} proved that the ring of integers of a real quadratic field $\mathbb{Q}(\sqrt{\Delta})$ satisfies \AFref{i} for all $i$ if and only if $\Delta \leq 21$.
\end{remark}
\begin{remark}
    There exist explicit examples of abelian varieties over number fields with good reduction everywhere.
    For example, the elliptic curves over $\mathbb{Q}(\sqrt{29})$ given by
    $$
    y^2 + xy +  \left(\frac{5+ \sqrt{29}}{2} \right)^2y =x^3
    $$
    found by Tate \cite[Section 5.2]{Serre1972} and over $\mathbb{Q}(\sqrt{6})$ given by
    $$
    y^2 +\sqrt{6}xy-y = x^3 - (2+\sqrt{6}x)
    $$
    found by Setzer \cite{Set81} have good reduction everywhere. 
\end{remark}

\begin{corollary} \label{cor: Minkowski}
The following rings satisfy \ref{item:Minkowski}, \Bref{i}, \PIDref{i}, and \AFref{i} for all $i$: $\mathbb{Z},\mathbb{Z}[i]$, and $\mathbb{Z}[\frac{1 + \sqrt{-3}}{2}]$.
\end{corollary}

\begin{proof}
By \autoref{thm: minkowski} and \cite[Proposition 1, p.417]{YamamuraImaginary}, these rings satisfy \ref{item:Minkowski} and hence \PIDref{i} for all $i$. By \autoref{thm: abrashkin}, they satisfy \AFref{i} for all $i$. As they are rings of integers of $\mathbb{Q}$ or an imaginary quadratic number field, they satisfy \Bref{i} for all $i$ by \autoref{cor: Brauer_group_integers}.
\end{proof}

\subsection{Picard schemes}
We will need some facts about Picard schemes of smooth and proper schemes over $S$.

\begin{lemma} \label{lem: global_section}
Assume that $S$ satisfies \ref{item:Minkowski}.
    Let $f \colon X \to S$ be a proper smooth morphism, where $X$ is integral. 
    Then, $f^{\sharp}:\mathcal{O}_S \to f_* \mathcal{O}_X$ is an isomorphism. 
    In particular, all fibers of $f$ are geometrically integral.
\end{lemma}

\begin{proof}
    Let $f \colon X \xrightarrow{g} \Spec  (f_* \mathcal{O}_X)\xrightarrow{h} S$ be the Stein factorisation of $f$. 
    As explained in \cite[Remarques (7.8.10)]{EGA_III_2}, $h$ is a finite \'etale morphism.
    As $\Gamma(X, \mathcal{O}_X)$ is integral and $S$ is \'etale simply connected, it follows that $f^{\sharp}$ is an isomorphism.
\end{proof}

By \cite{Raynaud}, the relative Picard functor behaves particularly well for morphisms with geometrically integral fibers over Dedekind domains, and we collect the corresponding results in the following:

\begin{proposition} \label{prop: pic}
Assume that $S$ satisfies \ref{item:Minkowski}. Let $f: X \to S$ be a proper smooth morphism, where $X$ is integral. Then, the following hold:
\begin{enumerate}
\item The \'etale sheafification of the relative Picard functor is representable by a group scheme $\Pic_{X/S}$ which is locally of finite type and separated over $S$ and satisfies the valuative criterion for properness.
\item The subsheaves $\Pic^{0}_{X/S} \subseteq \Pic_{X/S}$ (resp. $\Pic^{\tau}_{X/S} \subseteq \Pic_{X/S}$) of sections whose restriction to geometric fibers lies in the identity component (resp. the torsion component) are separated open subgroup schemes of $\Pic_{X/S}$ of finite type over $S$.
\item If $h^2(X_s,\mathcal{O}_{X_s}) = 0$ for all $s \in S$ closed, then $\Pic_{X/S}$ is smooth over $S$ and $\Pic^0_{X/S}$ is an Abelian variety of relative dimension $\frac{1}{2}b_1(X_K)$ over $S$.
\item If $b_1(X_K) = 0$, then $\Pic^{\tau}_{X/S}$ is finite over $S$.
\end{enumerate}
\end{proposition}

\begin{proof}
Since $f$ is proper, flat, and with geometrically reduced fibers, it is cohomologically flat in degree $0$ \cite[Proposition (7.8.6)]{EGA_III_2}, that is, formation of $f_* \mathcal{O}_X$ commutes with base change. 
Since $f$ is smooth, it has sections \'etale locally, so the \'etale sheafification $R^1f_{\et,*} \mathbb{G}_m$ of the relative Picard functor is representable by an algebraic space $\Pic_{X/S}$ locally of finite presentation over $S$ \cite[Theorem 7.3]{Artin}.
In fact, as the fibers of $f$ are smooth, $\Pic_{X/S} \to S$ is separated \cite[Section 8.4, Theorem 3]{NeronModels} and it satisfies the valuative criterion of properness over $S$.
Since separated and locally of finite type group objects in the category of algebraic spaces over Dedekind schemes are schemes \cite[Theorem 3.3.1]{Raynaud}, we deduce that $\Pic_{X/S}$ is a group scheme locally of finite type over $S$.

Now, consider the subfunctors $\Pic^{0}_{X/S}$ and $\Pic^{\tau}_{X/S}$ which are open subgroup schemes of $\Pic_{X/S}$ by 
\cite[Proposition (3.3.6)]{Raynaud}. By \cite[Exp. XIII, Thm. 4.7]{SGA6}, both of these group schemes are of finite type over $S$. This is Claim (2).

Claim (3) follows from \cite[Section 8.4, Proposition 2]{NeronModels}, $b_1(X_K) =b_1(X_{s}) = 2 \dim \Pic_{X_{s}/k(s)} = 2 \dim (\Pic_{X/S})_{s}$, and Claim (2).

Claim (4) follows from Claims (1) and (2) and the equality $b_1(X_K)  = 2\dim(\Pic_{X/S})_{s}$.
\end{proof}

The most important fact on Picard schemes for us is the following observation on constancy of Picard ranks.

\begin{theorem} \label{thm: full_Pic}
Assume that $S$ satisfies \ref{item:Minkowski}.
    Let $f \colon X \to S$ be a proper smooth morphism, where $X$ is integral.
    Assume that $h^2(X_{s},\mathcal{O}_{X_{s}}) = 0$ for all $s \in S$ closed.
    Then, the N\'eron--Severi sheaf
    ${\rm NS}_{X/S} \coloneqq \Pic_{X/S}/\Pic^0_{X/S}$ is constant, that is, ${\rm NS}_{X/S} = \underline{{\rm NS}(X_{\overline{K}})}$.

    In particular, all geometric fibers of $f$ have the same Picard rank and if $b_1(X_K) = 0$, then $\Pic_{X/S}$ is constant.
\end{theorem}

\begin{proof}
By \autoref{prop: pic} (3), $\Pic_{X/S}^0$ is an Abelian variety over $S$. By \cite[Tag 06PH]{stacks-project}, this implies that the N\'eron--Severi sheaf is representable by an algebraic space (hence by a scheme \cite[Theorem 3.3.1]{Raynaud}) and the quotient map $\Pic_{X/S} \to {\rm NS}_{X/S}$ is a $\Pic_{X/S}^0$-torsor.

By \autoref{prop: pic} (1) and (3), the morphism $\Pic_{X/S} \to S$ is smooth and satisfies the valuative criterion for properness and so does the $\Pic^0_{X/S}$-torsor $\Pic_{X/S} \to {\rm NS}_{X/S}$. Hence, the same holds for ${\rm NS}_{X/S} \to S$. As ${\rm NS}_{X/S} \to S$ is also of relative dimension $0$, it is \'etale, hence \'etale locally constant. 
By the theorem of the base, ${\rm NS}_{X/S}$ is a sheaf of Abelian groups of finite type over $S$.
As $\pi_1(S) = 0$, \cite[Tag 0GIY]{stacks-project} implies that ${\rm NS}_{X/S}$ is constant, as claimed. 
It remains to observe that if $k$ is algebraically closed with a morphism $\Spec(k) \to S$, then $H^1_{\et}(k,\Pic_{X/S}^0) = 0$, hence ${\rm NS}_{X/S}(k) = \Pic(X_k)/\Pic^0(X_k) = {\rm NS}(X_k)$.
\end{proof}

\begin{remark}
One cannot expect that ${\rm NS}_{X/S}(S) = \Pic_{X/S}(S)/\Pic^0_{X/S}(S)$, even if $S$ satisfies \ref{item:Minkowski}.

To have an explicit example, consider the elliptic curve $E_0$ over $\mathbb{Q}(\sqrt{5},\sqrt{13})$ named 1.1-a1 in \cite{lmfdb}. The curve $E_0$ has everywhere good reduction and, by the Birch--Swinnerton-Dyer conjecture, its Tate--Shafarevich group $\Sha(E_0)$ is expected to be of order $4$. Moreover, one can prove using discriminant bounds that the ring of integers $R$ of $\mathbb{Q}(\sqrt{5},\sqrt{13})$ satisfies \ref{item:Minkowski}.

Now, take a genus $1$ curve $X$ over $S = \Spec R$ corresponding to a non-trivial class in $\Sha(E_0)$. Then, ${\rm NS}_{X/S} \cong \underline{\mathbb{Z}}$ by \autoref{thm: full_Pic}. But since $X$ has no $S$-point, the inclusion $\Pic_{X/S}(S)/\Pic^0_{X/S}(S) \hookrightarrow {\rm NS}_{X/S}(S) = \mathbb{Z}$ is strict. More precisely, the image of $1$ under the boundary map ${\rm NS}_{X/S}(S) \to H^1_{\et}(S,E)$, where $E$ is the smooth model of $E_0$ over $S$, is exactly the class of $X$ in $\Sha(E_0) \subseteq H^1_{\et}(S,E)$.
\end{remark}

\subsection{Curves}
We recall the classification of smooth proper curves over $\Spec(\mathbb{Z})$ and see which properties of the previous section are used. The arithmetic genus of a curve $C$ over a field $k$ is $p_a(C)=h^1(C, \mathcal{O}_C).$

\begin{corollary} \label{cor: pa_constant}
Assume that $S$ satisfies \ref{item:Minkowski}. Let $f: X \to S$ be a proper smooth morphism, where $X$ is integral and $\dim(X) = 2$. Then, $p_a(X_s)$ is the same for all closed points $s \in S$.
\end{corollary}

\begin{proof}
    As $X$ is flat over $S$ we have that $\chi(X_s, \mathcal{O}_{X_s})$ is constant. By \autoref{lem: global_section}, the $X_{s}$ are geometrically integral, hence $h^0(X_s,\mathcal{O}_{X_s}) = 1$ and thus $p_a(X_s)$ is constant.
\end{proof}

\begin{lemma} \label{lem: smooth_curves_g0}
Assume that $S$ satisfies \ref{item:Minkowski}.
    Let $f \colon X \to S$ be a proper smooth morphism, where $X$ is integral and $\dim(X)=2$. If $p_a(X_K) = 0$, then the following hold:
\begin{enumerate}
\item $f$ is a smooth conic bundle.
\item If $S$ satisfies \Bref{2}, then $X \simeq \mathbb{P}_S(E),$ where $E$ is a locally free sheaf of rank 2 on $S$.
 \item If $S$ satisfies \Bref{2} and \PIDref{2}, then $X \simeq \mathbb{P}^1_S$.
\end{enumerate}
\end{lemma}

\begin{proof}
By \autoref{cor: pa_constant}, all fibers of $f$ have arithmetic genus $0$. By cohomology and base change, $\mathcal{F} \coloneqq f_* \mathcal{O}_X(-K_{X/S})$ is a locally free sheaf of rank $3$ and the anti-canonical map $X \hookrightarrow \mathbb{P}_S(\mathcal{F})$ realizes $f$ as a smooth conic bundle over $S$.
If $S$ satisfies \Bref{2}, the associated Brauer class is trivial and $X \simeq \mathbb{P}_S(E)$ for some locally free sheaf $E$ of rank $2$. If $S$ satisfies \PIDref{2}, then $E$ is trivial and then $X \simeq \mathbb{P}^1_S$. 
\end{proof}

\begin{proposition} \label{prop:curve_Z}
Assume that $S$ satisfies \ref{item:Minkowski}, \Bref{2}, \PIDref{2}, and \ref{item:AbrashkinFontaine} for all $i$.
    Let $f \colon X \to S$ be a proper smooth morphism, where $X$ is integral and $\dim(X)=2$.
    Then $X \simeq \mathbb{P}^1_{S}$.
\end{proposition}

\begin{proof}
Consider the Jacobian $\Pic^0_{X/S} \to S$. By \autoref{prop: pic} and since $h^2(X_{s},\mathcal{O}_{X_{s}}) = 0$, as the fibers are curves, this is an Abelian variety of dimension $\frac{1}{2}b_1(X_K)$. As $R$ satisfies \ref{item:AbrashkinFontaine}, we deduce that $b_1(X_K) = 0$. Hence, $p_a(X_K) = 0$ and so $X \simeq \mathbb{P}^1_S$ by \autoref{lem: smooth_curves_g0}.
\end{proof}

\begin{remark}
By \autoref{cor: Minkowski}, the classification result of
\autoref{prop:curve_Z} applies in particular to $R \in \{\mathbb{Z},\mathbb{Z}[i],\mathbb{Z}[\frac{1 +\sqrt{-3}}{2}]\}$.
\end{remark}

By considering Jacobians of curves, we can also prove the following stronger version of \autoref{rmk: AF_implies_mixed}:

\begin{proposition}
If $S$ satisfies \AFref{i} for all $i$, then $S \to \Spec(\mathbb{Z})$ is surjective.
\end{proposition}
\begin{proof}
Passing to Jacobians as in \autoref{prop:curve_Z}, it suffices to produce a smooth proper curve of positive genus over $\mathbb{Z}\big[\frac{1}{p}\big]$ for every prime $p$.
For $p = 2$, we can use the elliptic curve with Weierstra\ss ~equation $y^2 = x^3 - x$. For odd $p$, we can use the Fermat curve given in $\mathbb{P}^2_{\mathbb{Z}}$ by $(x^p + y^p + z^p = 0)$.
It is smooth over $\mathbb{Z}\big[\frac{1}{p}\big]$ by the Jacobian criterion and its genus is positive, since $p \geq 3$.
\end{proof}

\begin{question}
Is there an $i > 0$ such that \AFref{i} implies that $S \to \Spec (\mathbb{Z})$ is surjective? In other words, is there an $i$ such that for every prime $p$ and every Dedekind scheme $S$ without residue fields of characteristic $p$, there is a smooth family of abelian varieties of relative dimension $i$ over $S$?
\end{question}

\section{On the MMP for smooth projective models of surfaces} \label{section:MMP}

The Minimal Model Program (MMP) for semi-stable threefolds over a Dedekind domain with perfect residue fields of arbitrary characteristic was established in \cite{HW22, TY23}.
In this section we show how to quickly reprove those results in the case of families of smooth surfaces and its applications to the classification problems of smooth surfaces over $\mathbb{Z}$ and more general Dedekind schemes. Recall our convention for the base scheme $S$:

\begin{notation}
\label{notation:MMP}
In this section, $S$ is an excellent Dedekind scheme whose residue fields at closed points are perfect.
\end{notation}

We start with the base-point-free theorem for families of klt surface pairs.

\begin{proposition}[Easy base point free theorem] \label{prop: easy_bpf}
    Let $f \colon (X,\Delta) \to S$ be a projective contraction of relative dimension 2 such that $K_X+\Delta$ is $\mathbb{Q}$-Cartier and every closed fibre $(X_s, \Delta_s)$ is klt.
    Let $L$ be a nef Cartier divisor such that $L-(K_X+\Delta)$ is big and nef.
    Then 
    \begin{enumerate}
        \item[(i)] there exists $m_0>0$ such that $|mL|$ is base point free for every $m \geq m_0$;
        \item[(ii)] if $\varphi \colon X \to Y$ is the semiample contraction induced by $L$ and $t$ is a closed point of $S$, then the base change $\varphi_t \colon X_t \to Y_t$ is the semiample contraction induced by $L|_{X_t}$.
    \end{enumerate}
\end{proposition}

\begin{proof}
   
    By the base-point-free theorem for klt surfaces over fields (if $L_{X_t}$ is not numerically trivial it follows from \cite[Theorem 4.2]{Tan18} and if it is numerically trivial it follows from the vanishing $h^1(X_t,\mathcal{O}_{X_t}) =h^2(X_t,\mathcal{O}_{X_t})= 0$ for del Pezzo surfaces over perfect fields), there exists $n_0(t)>0$ such that for all $m \geq n_0(t)$ we have that the linear system $|mL_{X_t}|$ is base point free.
    Without loss of generality, we can suppose $n_0(t)$ is sufficiently large such that the morphism induced by $|mL_{X_t}|$ is the semiample contraction associated to $L_{X_t}$ by \cite[Theorem 2.1.27]{Laz1} for all $m \geq n_0(t)$.
    Moreover, possibly further enlarging the integer $n_0(t)$, we can suppose that $H^1(X_t, \mathcal{O}_{X_t}(mL_{X_t}))=0$ for all $m \geq n_0(t)$ by the asymptotic version of the Kawamata--Viehweg vanishing theorem (see \cite[Theorem 3.8]{Tan18}). 
    Consider the standard short exact sequence:
    $$0 \to \mathcal{O}_X(mL-X_t) \to \mathcal{O}_X(mL) \to \mathcal{O}_{X_t}(mL_{X_t}) \to 0.$$
    As $H^1(X_t, \mathcal{O}_{X_t}(mL_{X_t}))=0$, we deduce that the stalk $R^1f_*\mathcal{O}_X(mL)_t$ at the point $t$ vanishes by Grauert's theorem and thus also $R^1f_*\mathcal{O}_X(mL-X_t)_t \simeq R^1f_*\mathcal{O}_X(mL)_t \otimes \mathfrak{m}_t=0$. Therefore 
        $f_*\mathcal{O}_X(mL)_t \to H^0(X_t, \mathcal{O}_{X_t}(mL_{X_t}))$ is surjective.
    By Noetherian induction on $S$, we can find a natural integer $m_0>0$ such that for every $t \in S$ the map 
    \begin{equation} \label{eq: surj}
        f_*\mathcal{O}_X(mL)_t \to H^0(X_t, \mathcal{O}_{X_t}(mL_{X_t})) \text{ is surjective}
    \end{equation} for all $m \geq m_0$.
    Using this, we deduce that (i) holds, i.e. $|mL|$ is base point free for $m \geq m_0$.

    Let $\varphi \colon X \to Y$ be the semiample contraction associated to $L$.
    Combining the surjectivity in \autoref{eq: surj} with \cite[Lemma 2.17]{BBS24} we deduce also (ii): the base change $\varphi_t \colon X_t \to Y_t$ coincides with the semiample contraction associated to $L|_{X_t}$.
\end{proof}

\begin{remark}
    If $S$ has positive or mixed characteristic, the semiampleness part (i) in \autoref{prop: easy_bpf} could be proven by combining the base point free theorem of \cite[Theorem 3.8]{Tan18} with \cite{CTsemiampleness, Wit24} to deduce the statement that $|mL|$ is base point free for every \emph{sufficiently divisible} $m \geq m_0$, which is weaker than (i) (cf.~\cite{Tan20, Ber21} for the case of threefolds in characteristic $p$). 
    Such a proof has the advantage to generalise to families of klt 3-folds, despite the failure of the Kawamata--Viehweg vanishing theorem.
    Note that such a proof does not imply (ii) and in general there are counterexamples to (ii) in positive characteristic starting from dimension 4 \cite{Brivio25}.
\end{remark}

The cone theorem in our setting follows from the cone theorem for smooth surfaces over a field together with a deformation theory argument of extremal $K_X$-negative curves on smooth surfaces. This result is known from the work of Kawamata \cite[Theorem 1.3]{Kaw94} and Takamatsu--Yoshikawa \cite[Proposition 4.4.(2)]{TY23}.

\begin{theorem}[Easy cone theorem] \label{thm: cone}
     Let $f \colon X \to S$ be a projective smooth morphism of relative dimension 2.
    Then 
    \begin{equation} \label{eq:cone}
         \NE(X/S) = \NE(X/S)_{K_X \geq 0} + \sum_{i \in I} \mathbb{R}_{+}[C_i],
    \end{equation}
   where $I$ is a countable set and $\mathbb{R}_{+}[C_i]$ are extremal rays spanned by integral proper curves $C_i$ such that $K_X \cdot C_i<0$. 
   
   If $A$ is an ample $\mathbb{Q}$-Cartier $\mathbb{Q}$-divisor, then there exists a finite number of curves $C_1, \dots, C_{n_A}$ such that
   $$ \NE(X/S) = \NE(X/S)_{K_X+A \geq 0} + \sum_{i=1}^{n_A} \mathbb{R}_{+}[C_i].$$
\end{theorem}

We prove that smoothness descends along each step of a $K_X$-MMP for smooth projective families of surfaces over $S$.

\begin{theorem} \label{thm: smooth_after_MMP}
    Let $f \colon X \to S$ be a projective morphism of relative dimension 2 such that all fibres are regular (resp.~canonical).
    Let $\varphi \colon X \to Y$ be the contraction of a $K_X$-negative extremal ray $\xi=\mathbb{R}_{+}[C]$. 
    Then either 
    \begin{enumerate}
        \item $\varphi$ is a birational divisorial contraction and the fibres of the induced morphism $f' \colon Y \to S$ are regular (resp.~canonical);
        \item $\varphi$ is a Mori fibre space and the fibres of $Y \to S$ are regular.
    \end{enumerate}
\end{theorem}

\begin{proof}
Let $t$ be the closed point of $S$ such that $f(C)=t$.
By \autoref{prop: easy_bpf}.(ii), we know that $Y_t$ is the normal variety obtained as the contraction from $X_t$ of a $K_{X_t}$-negative extremal face.
Suppose $\varphi$ is a birational contraction.
In this case, we deduce that $Y_t$ is terminal (resp. canonical) of dimension 2 and thus regular (resp. Gorenstein) by \cite[Theorem 2.29]{kk-singbook}. 
In particular, in both cases $Y$ is Gorenstein and thus $\varphi$ is a divisorial contraction. In the first case, the $Y_t$ are terminal and hence $Y \to \Spec(R)$ is a smooth morphism. 

If $\varphi$ is not birational, then it is a Mori fibre space. If $\dim(Y)=1$, then $Y=S$ and the fibres are all regular (resp. canonical). If $\dim(Y)=2$, the fibres $Y_t$ over $S$ are normal curves by \autoref{prop: easy_bpf}.(ii) and thus regular.
\end{proof}

\begin{remark}
    Choosing the class of a curve $C$ which generates a $K_X$-negative extremal ray (instead of a curve $C$ such that its class $C_K$ in $X_K$ generates a $K_{X_K}$-negative extremal ray) for running the MMP in families is fundamental. 
    Consider the following example:
    Choose a prime $p$ and let $X$ be the blow-up of $\mathbb{P}^2_{\mathbb{Z}}$ at $[0:0:1]$ and at the strict transform of $[p:0:1]$ with exceptional divisors $E_1$ and $E_p$, respectively. 
    Note that the generic $E_{1,0}$ is a $K_{X_0}$-negative extremal ray. However $E_1$ cannot be contracted by a relative $K_X$-negative extremal contraction over $\mathbb Z$. Indeed, the specialisation of $E_{1,0}$ to $X_p$ is the union of a $(-1)$-curve and a $(-2)$-curve intersecting transversally. 
\end{remark}

Combining \autoref{thm: cone} together with \autoref{prop: easy_bpf}, we can run the MMP for projective models of smooth surfaces over Dedekind schemes with perfect residue fields.

\begin{theorem}[Running the MMP] \label{thm: running_MMP}
    Let $X \to S$ be a projective smooth morphism of relative dimension 2. 
    Then we can run a $K_X$-MMP over $S$:
    $$ X \coloneqq X_0 \xrightarrow{\varphi_1} X_1 \xrightarrow{\varphi_2} X_2 \to \dots \xrightarrow{\varphi_n} X_n =:Y.  $$
   Here each $X_i \to S$ is a smooth projective morphism, and the algorithm terminates with $Y \to S$ such that one of the following two cases holds:
    \begin{enumerate}
        \item there exists a Mori fibre space structure $Y \to Z$ over $S$, where $Z$ is smooth and projective over $S$;
        \item $K_Y$ is nef over $S$.
    \end{enumerate}
    Moreover, each $\varphi_i$ is a blow-up of a regular subscheme $Z_{i} \subset X_{i}$ that is \'etale over $S$.
\end{theorem}

\begin{proof}
    The cone theorem \autoref{thm: cone}  implies that for any ample Cartier divisor $A$, its nef threshold
    $\lambda_A \coloneqq \inf \left\{ \lambda \geq 0 \mid K_X+\lambda A \text{ is nef}\right\}$ is a non-negative rational number. 
    
    If $K_X$ is nef, we stop as we are in case (b).
    If $K_X$ is not nef, then $\lambda_A >0$.
    If $A$ is chosen general enough, by \autoref{thm: cone} we can suppose there exists an extremal curve $C$ such that $\NE(X) \cap (K_X+\lambda_A A)^{\perp} = \mathbb{R}_{+}[C]$.
    By \autoref{prop: easy_bpf}, the $\mathbb{Q}$-divisor $K_X+\lambda_A A$ is semi-ample and thus we can contract $C$ via $\varphi_1 \colon X \to X_1$. 
    If $\varphi_1$ is of fibre type, we conclude by \autoref{thm: smooth_after_MMP} we are in case (a).
    Otherwise, by \autoref{thm: smooth_after_MMP}, $\varphi_1$ is a birational contraction and $X_1$ has smooth fibres and $\rho(X_1/S) < \rho(X_0/S)$.
    We can thus repeat the procedure asking if $K_{X_1}$ is nef. 
    This algorithm will terminate as the Picard rank as $\rho(X/S)$ is finite and $\rho(X_{i+1}/S)<\rho(X_i/S)$ at every step. 

    For the last statement, let $E_{i-1}$ be the exceptional divisor contracted by $\varphi_{i} \colon X_{i-1}\to X_i$. Consider the short exact sequence $0 \to \mathcal{O}_{X_{i-1}}(-E_{i-1})\to \mathcal{O}_{X_{i-1}} \to \mathcal{O}_{E_{i-1}} \to 0$ and its pushforward via $\varphi_i$.
    Let $Z_i$ be the subscheme on $X_i$ such that $\mathcal{I}_{Z_i}= \varphi_{i*} \mathcal{O}_{X_{i-1}}(-E_{i-1})$. 
    Let $i_t \colon \Spec k(t) \to S$ be a point of $S$ and denote by $X_{i-1,t}$ and $E_{i-1,t}$ the base change of $X_{i-1}$ and $E_{i-1}$ via $i_t$. 
    As $E_{i-1,t}$ is an exceptional curve of the first kind  in $X_{i-1} \times k(t)$ by \autoref{thm: smooth_after_MMP}, we have  by \cite[Chapter 9, Lemma 3.6]{Liu} that $R^1\varphi_{i,t *} \mathcal{O}_{X_{i-1} \times k(t)}(-E_{i,t})=0$ (and thus $R^1\varphi_{i*} \mathcal{O}_{X_{i-1}}(-E_{i-1})=0$) and $\varphi_{i,t *} \mathcal{O}_{X_{i-1} \times k(t)}(-E_{i,t})= \mathfrak{m}_{\varphi_{i,t}(E_{i,t})}$. 
    We thus deduce that $\varphi_* \mathcal{O}_{E_{i-1}}=\mathcal{O}_{Z_i}$ and $\varphi_* \mathcal{O}_{E_{i-1,t}} \cong \mathcal{O}_{Z_{i,t}}$.
    As the scheme $Z_i \times k(t)$ is smooth over $k(t)$ for every $t$ in $S$, we conclude that $Z_i$ is smooth over $S$, hence \'etale by dimension reasons.
\end{proof}

We specialise now to the case of \'etale  simply connected Dedekind bases.

\begin{corollary}\label{cor: red_smooth_models}
    Suppose $S$ satisfies \ref{item:Minkowski} and let $X \to S$ be a projective smooth morphism of relative dimension $2$. 
    Then, $X$ is the iterated blow-up of $S$-points from  a smooth $S$-scheme $Y$ which either admits a Mori fibre space structure with smooth base or is a minimal model. 
\end{corollary}

Let us observe that the projectivity assumption on $X \to S$ is necessary in order to run the MMP, even for smooth families of proper surfaces. 
Recall that for smooth surfaces over a field, properness is equivalent to projectivity by Zariski--Goodman \cite[Theorem 1.28]{Bad01}. 
However, this equivalence no longer holds for smooth proper families of surfaces over Dedekind domains.

\begin{example}
We construct an example of a smooth, proper family of surfaces $Y \to \Spec(\mathbb{Z})$ for which a $K_Y$-MMP cannot be run (and, in particular, $Y \to \Spec(\mathbb{Z})$ is not projective). This is inspired by Hironaka's example of a proper non-projective threefold over the complex numbers \cite[Example 3.4.1]{Har77}.

Let $X_0$ be the blow-up of $\mathbb{P}^2_\mathbb{Q}$ at $[1:0:0],[0:1:0],[0:0:1]$.
Pick two distinct prime numbers $p$ and $q$
and consider the three $\mathbb{Z}$-points $P_1 = [1:0:0],P_2 = [1:pq:0],P_3 =[1:pq:(pq)^2]$. 
Let $X'$ be the model of $X_0$ over $\mathbb{Z}\big[\frac{1}{q} \big]$ obtained by blowing-up $P_1$ in $\mathbb{P}^2_{\mathbb{Z}[\frac{1}{q}]}$, then the strict transform of $P_2$, and finally the strict transform of $P_3$.
The surface $X'_p$ obtained at the reduction mod $p$ is a weak del Pezzo surface of degree $6$ with a unique $(-1)$-curve $E_p$ (see \cite[Section 8.4.2]{Dol12}), which is the specialization of the exceptional divisor over $P_2$. 
We repeat the above construction away from the prime $p$ to get a model $X''$ of $X_0$ over $\mathbb{Z}\big[\frac{1}{p}\big]$. Over $\mathbb{Z}\big[\frac{1}{pq}\big]$, we can glue $X'$ and $X''$ along the (resolution of the) standard quadratic Cremona transformation
$$
[x:y:z] \mapsto [(pq)^2y(pqx-y)+z(pqy-z): (pq)^3y(pqx-y):(pq)^3(pqx-y)(pqy-z)]
$$

We show that we cannot run a $K_Y$-MMP. Since the generic fiber of $Y$ is obtained from $\mathbb{P}^2_{\mathbb{Q}}$ by blowing up $3$ rational points, we have $\rho(Y)=4$ by \autoref{thm: full_Pic}. As Mori fibre spaces of surfaces over an algebraically closed field have Picard rank at most 2, every extremal contraction of $Y$ must be birational by \autoref{thm: full_Pic}.

Thus, seeking a contradiction, assume there exists a birational contraction $\varphi \colon Y \to Y'$ with $Y'$ smooth and proper over $\mathbb{Z}$. As in \autoref{thm: running_MMP}, we may assume that $\varphi$ is the blow-up of $Y'$ along a closed subscheme $Z \subseteq Y'$ which is \'etale over $\mathbb{Z}$. Since $\mathbb{Z}$ satisfies \ref{item:Minkowski}, $Z$ is a section of $Y' \to \Spec(\mathbb{Z})$, hence the exceptional divisor $E$ of $\varphi$ is isomorphic to $\mathbb{P}^1_{\mathbb{Z}}$ and the fibers of $E$ are $(-1)$-curves in the fibers of $Y$.
However, the closure in $Y$ of any $(-1)$-curve on the generic fiber $Y_{0}$ has reducible fibers over either $p$ or $q$. Indeed, by construction the unique $(-1)$-curves on the fibers $Y_p$ and $Y_q$ are specializations of different $(-1)$-curves in $Y_{0}$, a contradiction.
\end{example}

\section{Smooth models of surfaces of negative Kodaira dimension} \label{sec: kod_-infty}
The goal of this section is to present the classification of smooth projective models of surfaces of negative Kodaira dimension over Dedekind schemes $S$ with trivial \'etale fundamental group. The following theorem gives the more refined statements we get under extra assumptions on $S$.

\begin{theorem} \label{thm: class_kod_negative}
Assume that $S$ satisfies \ref{item:Minkowski}.
    Let $X \to S$ be a projective smooth contraction of relative dimension 2 admitting a Mori fibre space structure $X \to B$. 
    Then, one of the following three cases occurs:
    \begin{enumerate}
        \item $X$ is a $\mathbb{P}^2$-bundle over $S$ in the \'etale topology. Moreover:
        \begin{enumerate}
        \item[(a)] If $S$ satisfies \Bref{3}, then $X \simeq \mathbb{P}_S(\mathcal{F})$ for a vector bundle of rank $3$ on $S$.
        \item[(b)] If $S$ satisfies \Bref{3} and \PIDref{3}, then $X \simeq \mathbb{P}^2_S$.
        \end{enumerate}
        \item \label{hirzebruch_case} $B$ is a smooth conic bundle over $S$ and $X$ is a smooth conic bundle over $B$. Moreover:
        \begin{enumerate}
        \item[(a)] If $S$ satisfies \Bref{2}, then both conic bundles are projectivizations of rank $2$ vector bundles.
        \item[(b)] If $S$ satisfies \Bref{2} and \PIDref{2}, then $B \simeq \mathbb{P}^1_S$ and $X \simeq \mathbb{P}_B(\mathcal{E})$, where $\mathcal{E}$ is a vector bundle of rank $2$ on $B$.
        \item[(c)] If $S$ is the spectrum of a Euclidean domain satisfying \Bref{2}, then (b) holds and $\mathcal{E}$ is an extension of two line bundles.
        \end{enumerate}
        \item $B$ is smooth over $S$, $p_a(B_K)>0$ and $X \to B$ is a smooth conic bundle.
    \end{enumerate}
\end{theorem}

\begin{proof}
As $K_{X_{s}}$ is not pseudoeffective by hypothesis, we deduce that $h^2(\mathcal{O}_{X_s})=0$ for every $s \in S$ by Serre duality.

If $B \simeq S$, then $b_1(X_{k(S)})=0$ (as $X_{k(S)}$ is a geometrically rational surface). 
By \autoref{thm: full_Pic} we deduce that $\rho(X_{\overline{k(s)}})=1$ and thanks to the classification of surfaces with a Mori fibre space structure over algebraically closed fields we conclude that a geometric fibre $X_{k(\overline{s})}$ is isomorphic to $\mathbb{P}^2_{\overline{k(s)}}$. 
Thus $X \to S$ is a $\mathbb{P}^2$-bundle in the \'etale topology. If $S$ satisfies \Bref{3}, then the associated Brauer class is trivial and if, additionally, \PIDref{3}, the associated vector bundle is trivial. This yields Claims (1), (1) (a), and (1) (b).

If the base $B$ has relative dimension 1 over $S$, the fibres $B_s$ are normal by \autoref{thm: smooth_after_MMP} and thus regular.
As $S$ satisfies \ref{item:Minkowski}, \autoref{thm: full_Pic} implies that ${\rm NS}_{X/S}$ is constant and thus every fibre $X_s \to B_s$ over $s \in S$ has geometric Picard rank $2$ and thus it is a $\mathbb{P}^1$-bundle over $B_{s}$. In particular, $X \to B$ is a smooth conic bundle and thus corresponds to a class $\alpha \in \Br(B)[2]$. In particular, if $p_a(B_K)>0$ we have shown that Claim (3) holds.

    We are left to show Claim (2), so we now suppose $p_a(B_K) = 0$. As $B \to S$ is smooth, the properties of $B$ in Claim (2) follow from \autoref{lem: smooth_curves_g0}. For Claim (2) (a), we observe that if $S$ satisfies \Bref{2}, then ${\rm Br}(\mathbb{P}_S(\mathcal{F}))[2] = 0$ for every vector bundle $\mathcal{F}$ of rank $2$ on $S$ by (the same proof as) \cite[Corollary 6.1.4]{brauergroupbook}. Claim (2) (b) is then immediate. If $S$ is the spectrum of an Euclidean domain, the description in terms of extension of two line bundles follows from \cite[second Corollary on page 326]{Han78}.
\end{proof}

\begin{remark}
    Apart from $\mathbb{Z}$, a typical geometric example when we can apply \autoref{thm: class_kod_negative} is $S= \mathbb{P}^1_k$ where $k$ is an algebraically closed field. 
   If the characteristic of $k$ is zero, then $S=\mathbb{A}^1_k$ is simply connected with trivial Brauer group and thus \autoref{thm: class_kod_negative} is applicable, unlike the characteristic $p>0$ case where the affine line is not simply connected.
\end{remark}

\subsection{Models of Hirzebruch surfaces over Dedekind domains}
In this section, we describe in more detail the surfaces of case (2) of \autoref{thm: class_kod_negative}. 
First, we observe that Hirzebruch surfaces that are fiberwise of the same type are actually Zariski locally constant.

\begin{lemma} \label{lem: constancy_Hirzebruch}
    Suppose that $R$ satisfies \PIDref{1} and let $S = \Spec R$.
    Let $\mathcal{E}$ be a rank $2$ vector bundle on $\mathbb{P}^1_S$ and let $\pi \colon X=\mathbb{P}_{\mathbb{P}^1_S}(\mathcal{E}) \to \mathbb{P}^1_S$ be its projectivisation.
    Let $n \geq 0$ and suppose that for all points $s \in S$, the fibre $X_s$ satisfies $X_s \simeq \mathbb{F}_n$.
    Then $X=\mathbb{P}_{\mathbb{P}^1_S}(\mathcal{E}) \simeq \mathbb{P}_{\mathbb{P}^1_S}(\mathcal{O}_{\mathbb{P}^1_S} \oplus \mathcal{O}_{\mathbb{P}^1_S}(-n))$. 
\end{lemma}
\begin{proof}
    As $S$ satisfies \PIDref{1}, we have $\Pic(\mathbb{P}^1_{S}) = \mathbb{Z}[\mathcal{O}(1)]$.
    We can replace $\mathcal{E}$ by a twist with a certain $\mathcal{O}_{\mathbb{P}^1_S}(l)$ and suppose, without loss of generality, that $\mathcal{E}_K \simeq \mathcal{O}_{\mathbb{P}^1_K} \oplus \mathcal{O}_{\mathbb{P}^1_K}(-n)$. 
    Consider the section $\Sigma_K$ of $\mathbb{P}_{\mathbb{P}^1_K}(\mathcal{E}_K)$ of self-intersection $(-n)$.
    The closure $\Sigma$ of $\Sigma_K$ is a divisor and for $s \in S$ a closed point, the fiber $\Sigma_s$ of $\Sigma$ is an effective divisor of self-intersection $(-n)$. As $X_s \simeq \mathbb{F}_n$, we can write
    $
    \Sigma_s \sim  aF_s + b\Sigma_s',
    $
    with $a,b \in \mathbb{Z} \geq 0$, $F_s$ a fiber of the ruling of $X_s$, and $\Sigma_s'$ the section of square $(-n)$. Now, by the constancy of intersection numbers and since $F_s$ is the specialization of a fiber of the ruling of $X_K$, we have $\Sigma_s^2 = -n$ and $\Sigma_s \cdot F_s = 1$. This implies that $a = 0$ and $b = 1$, that is, $\Sigma_s$ is linearly equivalent and hence equal to $\Sigma_s'$.

    In particular, $\Sigma$ is a section of $\pi$ and thus we have a natural associated quotient $\mathcal{E} \to \mathcal{L} \to 0$.
    As $\Pic(\mathbb{P}^1_{S}) = \mathbb{Z}[\mathcal{O}(1)]$ and $\mathcal{L}_K \simeq \mathcal{O}_{\mathbb{P}^1_K}(-n)$, we deduce that $\mathcal{L} \simeq \mathcal{O}_{\mathbb{P}^1_S}(-n)$.
    Since the kernel of $\mathcal{E} \to \mathcal{L}$ is invertible and generically of degree $0$, it is isomorphic to $\mathcal{O}_{\mathbb{P}^1_S}$.
    We have
    $\Ext^1(\mathcal{O}_{\mathbb{P}^1_S}(-n), \mathcal{O}_{\mathbb{P}^1_S})=H^1(\mathbb{P}^1_S, \mathcal{O}_{\mathbb{P}^1_S}(n)) = 0$ for $n \geq -1$, where the vanishing of the cohomology group follows from the assumption that $S$ is affine. We conclude that $\mathcal{E} \simeq \mathcal{O}_{\mathbb{P}^1_S} \oplus \mathcal{O}_{\mathbb{P}^1_S}(-n)$.
\end{proof}

Next, we want to write down explicit equations for Hirzebruch surfaces over $S$ in the affine case. As a first step, we have the following application of the Quillen--Suslin theorem \cite{quillen,Suslin}:

\begin{proposition}
Assume that $R$ satisfies \PIDref{1}. Let $\mathcal{E}$ be a rank $2$ vector bundle on $\mathbb{P}^1_R$ and let $\pi: X = \mathbb{P}_{\mathbb{P}^1_R}(\mathcal{E}) \to \mathbb{P}^1_R$ be its projectivisation. Then, there exists a closed immersion $X \hookrightarrow \mathbb{P}^1_R \times \mathbb{P}^3_R$ such that $\pi$ is induced by the first projection.
\end{proposition}
\begin{proof}
It suffices to prove that some twist $\mathcal{E}(n)$ can be generated by $4$ global sections. By \cite[Theorem 4]{quillen}, the restriction of $\mathcal{E}$ to a standard affine chart $D_+(x_i)$ of $\mathbb{P}^1_R$ is free, hence generated by two sections $s_{i1},s_{i2}$. Lifting the resulting $4$ local sections to global sections of a twist of $\mathcal{E}$, we get the result.
\end{proof}

In the case where $\pi$ has a section, which always holds if $R$ is Euclidean by \autoref{thm: class_kod_negative} (2) (c), we can reduce the number of generators to $3$ and find an explicit normal form of the resulting hypersurface.

\begin{lemma} \label{eq:hirzebruch}
Assume that $R$ satisfies \PIDref{1} and let $\pi: X = \mathbb{P}_{\mathbb{P}^1_{R}}(\mathcal{E}) \to \mathbb{P}^1_{R}$ be the projectivisation of a rank $2$ vector bundle. Further assume that $\pi$ admits a section $\Sigma$. 
Then, there exists an $n \geq 0$ and a closed immersion $X \hookrightarrow \mathbb{P}^1_{R} \times \mathbb{P}^2_{R}$ with image given by an equation of the form
$$
x_0^ny_0 + x_1^ny_1 + f(x_0,x_1)y_2 = 0,
$$
where $f(x_0,x_1) \in R[x_0,x_1]$ is a homogeneous polynomial of degree $n$ and $\Sigma$ is given by the equation $y_2 = 0$.
\end{lemma}
\begin{proof}
Since $\pi$ admits a section, there exists a short exact sequence
$$
0 \to \mathcal{L} \to \mathcal{E} \to \mathcal{N} \to 0,
$$
where $\mathcal{L}$ and $\mathcal{N}$ are line bundles.
Since $\Pic(\mathbb{P}^1_{R}) = \mathbb{Z}[\mathcal{O}(1)]$, we may assume, after replacing $\mathcal{E}$ by a suitable $\mathcal{E}(n)$, that $\mathcal{L} = \mathcal{O}_{\mathbb{P}^1_{R}}$ and $\mathcal{N} = \mathcal{O}_{\mathbb{P}^1_{R}}(n)$ for some $n \in \mathbb{Z}$.

If $n \leq 0$, then ${\rm Ext}^1(\mathcal{O}_{\mathbb{P}^1_R}(n),\mathcal{O}_{\mathbb{P}^1_R}) \simeq H^1(\mathbb{P}^1_R, \mathcal{O}_{\mathbb{P}^1_R}(-n))=0$, so $\mathcal{E} \simeq \mathcal{O}_{\mathbb{P}^1_R} \oplus \mathcal{O}_{\mathbb{P}^1_R}(n)$. Replacing $\mathcal{E}$ by a suitable twist and swapping the two summands, we get a representation of $\mathcal{E}$ as an extension as above but with $n \geq 0$.

If $n = 0$, then $\mathcal{E} \simeq \mathcal{O}_{\mathbb{P}^1_{R}} \oplus \mathcal{O}_{\mathbb{P}^1_{R}}$ and the claim is clear.

For the case $n > 0$, choose coordinates $x_0,x_1$ on $\mathbb{P}^1_{R}$. 
Then, the global sections $x_0^n,x_1^ n \in H^0(\mathbb{P}^1_{R},\mathcal{O}_{\mathbb{P}^1_{R}}(n))$ generate $\mathcal{O}_{\mathbb{P}^1_{R}}(n)$. Since $H^1(\mathbb{P}^1_{R},\mathcal{O}_{\mathbb{P}^1_{R}}) = 0$, these two global sections lift to global sections $s_0,s_1 \in H^0(\mathbb{P}^1_{R},\mathcal{E})$. Let $s_2 \in H^0(\mathbb{P}^1_{R},\mathcal{O}_{\mathbb{P}^1_{R}}) \hookrightarrow H^0(\mathbb{P}^1_{R},\mathcal{E})$ be a non-zero global section. Then, we have a short exact sequence
$$
0 \to \mathcal{K} \to \mathcal{O}_{\mathbb{P}^1_R}^{\oplus 3} \overset{(s_0,s_1,s_2)}{\longrightarrow} \mathcal{E} \to 0
$$
for a line bundle $\mathcal{K}$ and we see that $\mathcal{K} \simeq \mathcal{O}_{\mathbb{P}^1_R}(-n)$ by taking determinants. This realizes $X$ as the zero locus of the section $s$ of $\mathcal{O}_{\mathbb{P}^1_R}(n) \boxtimes \mathcal{O}_{\mathbb{P}^2_R}(1)$ corresponding to the inclusion $\mathcal{K} \to \mathcal{O}_{\mathbb{P}^1_R}^{\oplus 3}$. Choosing projective coordinates $y_0,y_1,y_2$ on $\mathbb{P}^2_R$, we can write the section $s$ for our choice of generators as
$$
s = (x_0^n + \lambda f(x_0,x_1)) y_0 + (x_1^n + \mu f(x_0,x_1)) y_1 + f(x_0,x_1)y_2,
$$
where $f(x_0,x_1)$ is a homogeneous polynomial of degree $n$ coming from the section $s_2$ and $\lambda,\mu \in R$ are scalars. Then, the substitution $y_2 \mapsto y_2 - \lambda y_0 - \mu y_1$ yields the normal form of the statement.
\end{proof}

In particular, we can find equations for all Hirzebruch surfaces over $\mathbb{Z}$ in $\mathbb{P}^1_\mathbb{Z} \times \mathbb{P}^2_{\mathbb{Z}}$. Note that a single Hirzebruch surface can have equations of various bidegrees $(1,n)$. It is an interesting and difficult problem to determine the minimal such $n$ from the vector bundle $\mathcal{E}$ (see e.g. \cite{Smirnov15}).

\begin{example}
Let $\mathcal{E}$ be a vector bundle of rank $2$ on $\mathbb{P}^1_{\mathbb{Z}}$. If $\mathcal{E}_0 \simeq \mathcal{O}_{\mathbb{P}^1_{\mathbb{Q}}}^{\oplus 2}$ and $\mathcal{E}_p \simeq \mathcal{O}_{\mathbb{P}^1_{\mathbb{F}_p}}(-i) \oplus \mathcal{O}_{\mathbb{P}^1_{\mathbb{F}_p}}(i)$ for $i \in \{0,1\}$ for all $p$, then by \cite{Smi16} there is a short exact sequence
$$
0 \to \mathcal{O}_{\mathbb{P}^1_\mathbb{Z}}(-2) \to \mathcal{E} \to \mathcal{O}_{\mathbb{P}^1_\mathbb{Z}}(2) \to 0
$$
and hence an equation of $\mathbb{P}(\mathcal{E})$ of bidegree $(1,4)$ in $\mathbb{P}^1_{\mathbb{Z}} \times \mathbb{P}^2_{\mathbb{Z}}$.
If $\mathcal{E}_0 \simeq \mathcal{O}_{\mathbb{P}^1_\mathbb{Q}} \oplus \mathcal{O}_{\mathbb{P}^1_\mathbb{Q}}(1)$ an analogous result could in principle be extracted from \cite{Iakovenko}.
\end{example}
\begin{example} \label{ex:equations} 
Let us spell out the equations of smooth models of Hirzebruch surfaces over the integers for small values of $n$ in the equation of \autoref{eq:hirzebruch}. 
Note that we can always use substitutions of the form $(y_0,y_1,y_2) \mapsto (y_0 + my_2, y_1 + m'y_2, y_2)$ for some integers $m,m'$ to kill the $x_0^n$ and $x_1^n$ terms in $f$.

\begin{itemize}
\item $n = 0$: In this case, $X \simeq \mathbb{P}^1_{\mathbb{Z}} \times \mathbb{P}^1_{\mathbb{Z}}$, embedded linearly in $\mathbb{P}^1_{\mathbb{Z}} \times \mathbb{P}^2_{\mathbb{Z}}$.
\item $n = 1$: Applying the above substitutions, we get the normal form
$$
x_0y_0 + x_1y_1 = 0.
$$
\item $n = 2$: In this case, the above substitutions lead to the normal forms
$$
x_0^2y_0 + x_1^2y_1 + m x_0x_1y_2 = 0
$$
for an integer $m$. 
\item $n \geq 2$: The normal form is 
$$ x_0^ny_0+x_1^ny_1+\big(\sum_{j=1}^{n-1} m_j x_0^jx_1^{n-j} \big)y_2 =0
$$
where $m_j$ are integer numbers.
\end{itemize}

\subsection{Elementary transformation of vector bundles on regular surfaces}

In this section we show \autoref{prop:Hirzebruch_surface_intro}. We show point \ref{item:parity} in \autoref{prop:types_of_Hirzebruch_surfaces}, point \ref{item:all_appear} in \autoref{ex:Hirzebruch_all_possible}, and point \ref{item:elementary_transf} in \autoref{prop:dec_elem_transf} and \autoref{prop:elementary_transf_to_blow_up}. 

\begin{notation}
\label{notation:space_elementary_transformation}
Consider the following situation.
\begin{itemize}
    \item $f : X= \bP^1_R \to R$ for a  Dedekind ring $R$,
    \item $\sE$ and $\sF$ are locally free sheaves of the same rank on $X$, 
    \item $Y= \bP_X(\sE)$ and $Z = \bP_X(\sF)$ are the associated projective bundles,
    \item $\pi: Y \to X$ and $\tau : Z \to X$ are the structure morphisms,  and
    \item $\eta$ is the generic point of $\Spec R$.
\end{itemize}
\end{notation}

\begin{definition}
\label{def:elementary_transformation}
Using \autoref{notation:space_elementary_transformation},  an \emph{elementary transformation} from $\sF$ to $\sE$ is a short exact sequence of the form
\begin{equation*}
\xymatrix{
0 \ar[r] & \sF \ar[r] & \sE \ar[r] & \iota_* \sL \ar[r] &  0
}
\end{equation*}
where 
\begin{itemize}
    \item $\iota : C \hookrightarrow X$ is an irreducible regular closed curve on $X$, 
    \item $\sL$ is a line bundle on $C$. 
\end{itemize}
We call $C$ the \emph{center} of the elementary transformation. 

We say that the elementary transformation is of \emph{fiber type} if $C= \bP^1_{x}$ for some closed point $x \in \Spec R$, and it is \emph{horizontal} otherwise. 
\end{definition}

\begin{remark}
Using \autoref{notation:space_elementary_transformation}, if $\sE$ is of rank $2$, then for every $x \in \Spec R$ we have 
\begin{equation*}
\sE|_{\bP^1_x}\simeq \sO_{\bP^1_x}(i) \bigoplus \sO_{\bP^1_x}(j),
\end{equation*}
where $\{i,j\}$ is uniquely determined. Hence, if we assume that $i\leq j$, then $j-i \geq 0$ is an invariant of $\sE$ at $x$. Let us say in this case that \emph{$\sE$ is of type $j-i$ at $x$} denoted $\type_x(\mathcal{E})$.

Similarly, we say that $Y$ has type $n$ at $x$ if $\sE$ has type $n$ at $x$. This is well-defined as $n$ is the degree of the Hirzebruch surface $Y_x \to \mathbb{P}^1_x$.
\end{remark}

\begin{remark}
For the analysis of the Hirzebruch surface $\mathbb{P}_R(\mathcal{E})$, we may replace $\mathcal{E}$ by a twist $\mathcal{E}(n)$. We call $\mathcal{E}$ \emph{normalized} if $\mathcal{E}_{\eta} \simeq \mathcal{O}_{\mathbb{P}^1_{\eta}}(-1) \oplus \mathcal{O}_{\mathbb{P}^1_{\eta}}(-n-1)$. Then, every $\mathcal{E}$ has a unique normalized twist.
\end{remark}

\begin{proposition} \label{prop:types_of_Hirzebruch_surfaces}
Using \autoref{notation:space_elementary_transformation}, assume that $\rk \sE=2$. Let $x \in \Spec(R)$. Then,
\begin{enumerate}
    \item  the difference ${\rm type}_x(\mathcal{E}) - {\rm type}_{\eta}(\mathcal{E})$ is an even non-negative integer, and
     \item if $\mathcal{E}$ is normalized, then ${\rm type}_x(\mathcal{E}) -  {\rm type}_{\eta}(\mathcal{E})= 2h^0(\mathbb{P}^1_x,\mathcal{E}|_{\mathbb{P}^1_x})$.
\end{enumerate}
\end{proposition}

\begin{proof}
For (1), note that the function $x \mapsto \deg \sE_x$ is constant, which implies that $\type_x (\sE) - \type_{\eta}(\sE)$ is even. For the non-negativity of this quantity, apply the semi-continuity of $x \mapsto h^0(X_x, \sE_x)$.

For (2), we may assume that $\mathcal{E}$ is normalized with $\mathcal{E}_{\eta} \simeq \mathcal{O}_{\mathbb{P}^1_{\eta}}(-1) \oplus \mathcal{O}_{\mathbb{P}^1_{\eta}}(-n-1)$. Then, the result is just a rephrasing of $h^0\left(\mathbb{P}^1_x,\sO_{\bP^1_x}(-1+j) \bigoplus \sO_{\bP^1_x}(-n-1-j)\right) = j$ for integer $j \geq 0$.
\end{proof}



In \autoref{ex:Hirzebruch_all_possible} we show that all possibilities that are left open by \autoref{prop:types_of_Hirzebruch_surfaces} for the type of $\sE$ at different primes   do appear.

\begin{proposition}
\label{prop:generic_injection_elementary_transformation}
Using \autoref{notation:space_elementary_transformation}, if $\sF \hookrightarrow \sE$ is an injection such that every irreducible components of $\Supp \left( \factor{\sE}{\sF} \right)$ is regular, then there is an intermediate vector bundle $\sF \subseteq \sG \subseteq \sE$ such that $\sF \subseteq \sG$ is an elementary transformation.
\end{proposition}

\begin{proof}
Set $\sH:=\factor{\sE}{\sF}$ and let $\pi: \mathcal{E} \to \mathcal{H}$ be the projection map. As both $\sF$ and $\sE$ are $S_2$, $\sH$ is $S_1$. Hence $\Supp \sH$ is pure of codimension $1$. For $i=1, \dots, s$, let $C_1,\dots,C_s$ be the irreducible components of $\Supp \sH$, and let $n_i$ be the smallest integer such that $\sO_X(-n_iC_i) \cdot \sH$ is zero at the generic point of $C_i$. Consider then 
$$
\sQ_{\pre}:=  \sO_X\big(-(n_1-1)C_1 \big) \cdot \prod_{i=2}^n \sO_X(-n_i C_i) \cdot \sF \subseteq \sH.
$$ Note that $\sQ_{\pre}$ is a coherent sheaf on $C_1$, possibly of higher rank, which is non-zero at the generic point of $C_1$. Let $\sQ$ be a rank $1$ coherent subsheaf of $\sQ_{\pre}$ on $C_1$. Let $\sG$ be the reflexivization of $\pi^{-1} ( \sQ)$. Note that as $\pi^{-1} ( \sQ)$ is a subsheaf of $\sE$, it is torsion-free. Hence, the natural homomorphism $\pi^{-1} ( \sQ) \to \sG$ is an isomorphism in codimension $1$, and in particular at the generic point $\eta_{C_1}$ of $C_1$. 

As $\sG$ is the reflexivization of a subsheaf of $\sE$, which is reflexive itself too, there is an induced homomorphism $\sG \to \sE$. This homomorphism is injective at the generic point of $X$, and it is between torsion-free sheaves, so it is injective globally. This justifies that we can think about $\sG$ as a subsheaf of $\sE$. Additionally, as $X$ is regular of dimension $2$, $\sG$ is locally free by \cite[Cor 1.4]{Hartshorne_Stable_reflexive_sheaves}. Finally, the quotient $\factor{\sG}{\sF}$ is $S_1$, by arguing as above. Hence, combining this with the choice of $\sQ$ and with the fact that $\big(\pi^{-1}(\sQ)\big)_{\eta_{C_1}} \to \sG_{\eta_{C_1}}$ is an isomorphism, $\factor{\sG}{\sF}$ is a torsion-free sheaf of rank $1$ on $C_1$. Now, using that $C_1$ is regular, by the classification of finitely generated modules over a PID, we obtain that  $\factor{\sG}{\sF}$ is a line bundle on $C_1$.
\end{proof}

\begin{proposition} \label{prop:dec_elem_transf}
Using the \autoref{notation:space_elementary_transformation}, every two vector bundles of rank $r$ on $X$ can be connected by a sequence of elementary transformations.
\end{proposition}

\begin{proof}
We know that $\sE_\eta \cong \bigoplus_{i=1}^r \sO_{\bP^1_{\eta}}(j_i)$. By twisting $\sE$ we may assume that $j_i\geq 0$ for all $i$. As $k(\eta)$ is infinite, we may choose global section $s_{i, \eta}$ of the summand $\sO_{\bP^1_{\eta}}(j_i)$ such that $\factor{\sO_{\bP^1_{\eta}}(j_i)}{\sO_{\bP^1_{\eta}}\cdot s_i}$ is a direct sum of skyscraper sheaves supported at distinct $k(\eta)$-rational points (and these points are even distinct for separate values of $i$). As we are working over an affine base $R$, after rescaling by elements of $R$, the above sections $s_{i, \eta}$ extend to global section $s_i \in H^0(X, \sE)$. This gives a generically isomorphic embedding $\sO_X^{\oplus r} \hookrightarrow \sE$ such that $\factor{\sE}{\sO_X^{\oplus r}}$ is supported on the union of some sections (not necessarily disjoint) and fibers (possibly with multiplicity) of $\bP^1_R \to \Spec R$. Hence, \autoref{prop:generic_injection_elementary_transformation} applies, and by Noetherian induction it shows that $\sE$ can be connected via elementary transformations to $\sO_X^{\oplus r}$.
\end{proof}

\begin{definition}
\label{def:birat_map_from_generically_isom_injection}
In the situation of \autoref{notation:space_elementary_transformation}, assume that there is an injection $\iota : \sF \hookrightarrow \sE$. Let $U$ be any open set over which $\iota$ is an isomorphism. Then, we define $\phi_\iota : \bP_X(\sE) \dashrightarrow \bP_X(\sF)$ be the rational map for which $\phi_\iota|_{\pi^{-1}U}$ is the morphism induced by $\iota|_{\pi^{-1}U}$. Note this is well-defined, that is, it is independent of $U$. 
\end{definition}

Note that if $S$ is an $\bN$-graded ring, then for every invertible $f \in S_0$, we obtain a ring automorphism $\xi_f : S \to S$ that sends $g \in S_i$ to $gf^i$. Note that $\xi_f$ induces on $\Proj S$ the identity morphism. In particular, using \autoref{notation:space_elementary_transformation}, if $\sL$ is a line bundle on $X$, then there is a canonical isomorphism $\xi_{\sL} :    \bP(\sE) \to \bP(\sE \otimes \sL)$. 

\begin{lemma}
\label{lem:projectivization_and_twist}
In the situation of \autoref{notation:space_elementary_transformation}, Let $C \subseteq X$ be a curve, and let $\iota: \sE(-C) \to \sE$ the natural injection. Then $\phi_\iota=\xi_{\sO_X(-C)}$ as rational maps. 
\end{lemma}

\begin{proof}
As rational maps are determined on non-empty open set, we may work over $U= X \setminus C$. Then, one just has to verify that over $\pi^{-1}U$, we have $\xi_{\sO_X(-C)}^{-1} \circ \phi_\iota= \id_{\sE(-C)}$, which is immediate from the definitions.
\end{proof}

\begin{lemma}
\label{lem:blow_up_diagonal}
Let $Y= \bP^1_{A,x,y}$ for some Noetherian ring $A$, and let $\alpha : \sO_Y u \oplus \sO_Y v \to \sO_Y(1)$ be the tautological evaluation homomorphism given by $\alpha(u)=x$ and $\alpha(v)=y$. Consider $Y= \bP_Y\big(\sO_Y(1)\big)$ as a closed subscheme of $\bP_Y(\sO_Y \oplus \sO_Y)= Y \times_A Y$ via the homomorphism $\alpha$. Then, $Y$ is the diagonal subscheme of $Y \times_A Y$.
\end{lemma}

\begin{proof}
This is a local computation. The key is that if we regard $\alpha$ as a homomorphism of graded modules over $A[x,y]$, then $\ker \alpha= (yu-xv)$. Hence, the equation of $Y$ in  $Y \times_A Y = \bP^1_{A,x,y} \times_A \bP^1_{A,u,v}$ is $(yu-xv)$, which is the defining equation of the diagonal. 
\end{proof}

\begin{proposition}
\label{prop:elementary_transf_to_blow_up}
Using the \autoref{notation:space_elementary_transformation}, let 
\begin{equation*}
\xymatrix{
0 \ar[r] & \sF \ar[r] & \sE \ar[r] & \iota_* \sL \ar[r] & 0 
}
\end{equation*}
be an elementary transformation on $X$ of rank $2$ vector bundles. Let $V$ be the section of $\bP_X(\sE)|_C \to C$  corresponding to the surjection $\sE|_C \twoheadrightarrow \sL$. We think about $V$ as a closed subscheme of $\bP(\sE)$. 
Then, there is a section $U$ of $Z:=\bP_X(\sF)$ over $C$ such that $\Bl_V Y \simeq \Bl_U Z $.
\end{proposition}

\begin{proof}
 Let $\alpha : \sF \hookrightarrow \sE$ be the injection given by the considered elementary transformation.  Let $W$ be the closure of the graph of $\phi_\alpha$ in $Y \times_X Z$.

In the rest of the proof our goal is to show that $W$ has the description as blow-ups of sections of $\pi|_{\pi^{-1}(C)}$ and of $\tau|_{\tau^{-1}(C)}$, respectively. We start with the first one. Consider the following composition:
\begin{equation*}
\xymatrix{
\pi^* \sF \ar@{^(->}[r] \ar@/^2pc/[rr]^{\xi} & \pi^* \sE \ar[r] & \sO_Y(1)
}
\end{equation*}
where $\pi^*\mathcal{E} \to \mathcal{O}_Y(1)$ is the natural homomorphism associated to $Y = \mathbb{P}(\mathcal{E})$.
We have $\im \xi =  \sO_Y(1) \otimes \sI_{V/Y} $.
Hence, if $\sI:=\sI_{V/Y}$, then 
\begin{equation*}
\Bl_V Y = \Proj \left( \bigoplus_n \sI^n \right) = \Proj \left( \bigoplus_n \sO_Y(n) \otimes \sI^n \right)
\end{equation*}
is a closed subscheme of 
\begin{equation*}
\bP_Y(\pi^* \sF) \cong Y \times_X Z
\end{equation*}
Note that over $X \setminus C$, $\xi$ is surjective and so on this open set we can identify $\sF$ and $\sE$. Then we can apply \autoref{lem:blow_up_diagonal}, which gives that via this identification $\Bl_V Y$ is the diagonal over $X \setminus C$. However, by \autoref{def:birat_map_from_generically_isom_injection}, the same holds for $W$. As $W$ and $\Bl_V Y$ are both integral closed subschemes of $Y \times_X Z$ and they agree on a non-empty open set, they have to be equal. This concludes the first description of $W$ as a blow-up of a section of $\pi|_{\pi^{-1}(C)}$

Now, we work on the second description of $W$ as a blow-up of a section of $\tau|_{\tau^{-1}(C)}$ on $Z$. First, we have to define the corresponding section of $\tau|_{\tau^-1(C)}$. Note that as $\sL$ is locally free on $C$, $\iota_* \sL$ is annihilated by $\sO_X(-C) \subseteq \sO_X$. Hence, $\sE(-C) \subseteq \sF$. Let us call $\beta$ this embedding. Additionally, as $\factor{\sE}{\sE(-C)}$ is a rank $2$ vector bundle on $C$, $\factor{\sF}{\sE(-C)}$ is a line bundle on $C$. Hence, we obtain an elementary transformation from $\sE(-C)$ to  $\sF$. Let $U$ be the section of $\tau|_{\tau^{-1}C}$ corresponding to this elementary transformation.   As in the previous case, using \autoref{lem:blow_up_diagonal},  we get that $\Bl_U Z$ is the closure of the graph of $\phi_\beta$ in $\bP_X(\sE(-C)) \times_X Z$. However, $W$ is defined as the closure of the graph of $\phi_\alpha$, not of $\phi_\beta$. Hence,
we are left to identify $\phi_\alpha$ with $\phi_\beta$ using the identification $\bP(\sE(-C)) \cong \bP(\sE) = Y$. By \autoref{lem:projectivization_and_twist}, this latter identification is $\phi_\gamma^{-1}$, where $\gamma$ is the embedding $\gamma : \sE(-C) \hookrightarrow \sE$. Hence,  this identification indeed equates $\phi_\alpha$ and $\phi_\beta$, as we have $\gamma= \alpha \circ \beta$.
\end{proof}

\begin{example} 
\label{ex:Hirzebruch_all_possible}
Here we show that all types of Hirzebruch surfaces described by \autoref{prop:types_of_Hirzebruch_surfaces} exist. Fix a generic type $n$.
Find prime numbers $p_1,\dots, p_r$ and integers $n_i \geq 0$ for $i=1,\dots,r$. For each $i=1,\dots,r$, fix also surjections $\alpha_i : \sO_{\bP_{\bF_{p_i}}^1} (-1) \oplus \sO_{\bP_{\bF_{p_i}}^1} (-n -1) \twoheadrightarrow \sO_{\bP_{\bF_{p_i}}^1}(n_i-1)$. Compose then these surjections with the restriction homomorphisms
\begin{equation*}
\sF:=\sO_{\bP_{\bZ}^1} (-1) \oplus \sO_{\bP_{\bZ}^1}  (-n-1)\twoheadrightarrow \bigoplus_{i=1}^r \sO_{\bP_{\bF_{p_i}}^1} (-1) \oplus \sO_{\bP_{\bF_{p_i}}^1} (-n-1)
\end{equation*}
Call $\xi$ the obtained homomorphism. We claim that $\sE:=\ker \xi$ is a rank $2$ vector bundle on $\bP_{\bZ}^1$ such that 
\begin{equation*}
\sE|_{\bP_{\bF_p}^1} \cong
\left\{
\begin{matrix}
\sO_{\bP_{\bF_p}^1} (-1) \oplus \sO_{\bP_{\bF_p}^1}(-n-1) & \textrm{if }p \not\in \{p_1,\dots, p_r\} \\
\sO_{\bP_{\bF_p}^1}(n_i-1) \oplus \sO_{\bP_{\bF_p}^1}(-1-n-n_i)  & \textrm{if }p=p_i \\
\end{matrix}
\right.
\end{equation*}
As a consequence, $\bP(\sE|_{\mathbb{P}^1_{\mathbb{F}_p}})$ is a Hirzebruch surface of type $n$ at $p \not\in \{p_1,\dots,p_r\}$, and of type $n+2n_i$ at $p=p_i$.

To show our claim, by \autoref{prop:types_of_Hirzebruch_surfaces} it is enough to claim that for every $i=1,\dots,r$ we have
\begin{equation*}
h^0\left(\bP^1_{\bF_p}, \sE|_{\bP^1_{\bF_{p_i}}} \right) = n_i.
\end{equation*}\
Fix such an $i$, and write $p=p_i$ and $d=n_i$ and we work locally over $p$. 
Note that  $\sE|_{\bP^1_{\bF_p}} \cong \factor{\sE}{\sE\left(-\mathbb{P}^1_{\mathbb{F}_p}\right)}$ and that we have inclusions $\sE\left(-\mathbb{P}^1_{\mathbb{F}_p}\right) \subseteq \sF\left(-\mathbb{P}^1_{\mathbb{F}_p}\right) \subseteq \sE \subseteq \sF$.  Hence, we have an exact sequence
\begin{equation*}
\xymatrix{
0 \ar[r] & \factor{\sF\left(-\bP^1_{\bF_p}\right)}{\sE\left(-\bP^1_{\bF_p}\right)}  \ar[r] &  \sE|_{\bP^1_{\bF_p}} \ar[r] & \factor{\sE}{\sF\left(-\bP^1_{\bF_p}\right)} \ar[r] & 0
}
\end{equation*}
We have 
\begin{equation*}
\factor{\sF\left(-\bP^1_{\bF_p}\right)}{\sE\left(-\bP^1_{\bF_p}\right)} \cong \factor{\sF}{\sE}|_{\mathbb{P}^1_{\mathbb{F}_p}} \cong \sO_{\bP_{\bF_p}^1}(d-1).
\end{equation*}
As $\deg \sE|_{\bP^1_{\bF_p}}=-2-n$, this then also implies that \begin{equation*}
\factor{\sE}{\sF\left(-\bP^1_{\bF_p}\right)}\cong \sO_{\bP_{\bF_p}^1}(-n-1-d).
\end{equation*}
This concludes that $h^0\left(\sE|_{\bP^1_{\bF_p}} \right)=d $.

In particular, by \autoref{prop:elementary_transf_to_blow_up}, $\bP_{\mathbb{P}^1_{\mathbb{Z}}}(\sE)$ can be obtained from $\bP_{\mathbb{P}^1_{\mathbb{Z}}}(\sF)$ by blowing up the degree $n_i$ sections over $\bP^1_{\bF_{p_i}}$ corresponding to the surjections $\alpha_i$, and then contracting the strict transform of $\mathbb{P}_{\mathbb{P}^1_{\mathbb{F}_p}}(\mathcal{F}_p)$. 
\end{example}

\section{On the classification of smooth projective models of surfaces over $\mathbb{Z}$} \label{sec: smooth_Z}

In this section, we apply the previous results together with some recent results in the literature to study smooth projective surfaces over $\Z$. 
We leave it to the reader to spell out to which extent the results of this section generalize to more general bases satisfying the various conditions of \autoref{sec: integers}.

We recall that the Kodaira dimension is invariant in smooth families of surfaces in positive or mixed characteristic by \cite[Theorem 9.1]{KU-85}. 
Because of this, it is natural to subdivide the classification according to the Kodaira dimension of a fibre.
We now give a first partial classification of minimal models over $\mathbb{Z}$.

\begin{theorem} \label{thm: rough_classification_Z}
    Let $X \to \Spec(\mathbb{Z})$ be a projective smooth morphism, where $X$ is integral and $\dim(X)=3$. Then, $X$ is the successive blow-up at $\mathbb{Z}$-points of one of the following:
    \begin{enumerate}
        \item $\mathbb{P}^2_{\mathbb{Z}}$ or $\mathbb{P}_{\mathbb{P}^1_{\mathbb{Z}}}(\mathcal{E})$ where $\mathcal{E}$ is rank 2 vector bundle corresponding to an extension class in $\Ext^1(\mathcal{O}_{\mathbb{P}^1_{\mathbb{Z}}}, \mathcal{O}_{\mathbb{P}^1_{\mathbb{Z}}}(n))$ for some $n$;
        \item a $K_Y$-trivial fibration $Y \to \mathbb{P}^1_\mathbb{Z}$ such that $Y_{0}$ has Kodaira dimension 1;
        \item a smooth minimal surface $Y \to \Spec(\mathbb{Z})$ with big and nef canonical class $K_Y$ with $K_{Y_p}^2 =10-\rho(Y_p).$ 
        Moreover, $h^0(Y_p, \mathcal{O}_{Y_p}(mK_{Y_p}))=h^0(Y_0, \mathcal{O}_{Y_0}(mK_{Y_0}))$ for $m \geq 2$. 
        If $Y_{\can}$ is the canonical model of $Y$, then $Y_{\can} \times \Spec k(t)$ is the canonical model of $Y \times \Spec k(t)$ and $|5K_{Y_{\can}}|$ is a very ample linear system. 
        In particular, smooth canonically polarised surfaces form a bounded family over $\mathbb{Z}$.
    \end{enumerate}
\end{theorem}

\begin{proof}
    We divide the proof according to the Kodaira dimension of $X$. By \autoref{cor: red_smooth_models} we need to classify only Mori fibre spaces (case (1), done in \autoref{thm: class_kod_negative}) and minimal models $Y$ smooth over $\mathbb{Z}$.

    The case of Kodaira dimension 0 does not occur. Indeed, by combining the classification of smooth surfaces with trivial canonical class with \autoref{thm: abrashkin}, the only possible situation would be that the minimal model $Y$ of $X$ is a smooth family of Enriques surfaces over the integers. Such a family cannot exist by \cite{Sch23}.

    Let us consider a family of minimal surfaces of Kodaira dimension 1.
    By the abundance theorem for smooth surfaces, we know that $K_{Y_t}$ is semi-ample for every $t \in \Spec(\mathbb{Z})$ and thus $K_Y$ is semiample over $\mathbb{Z}$ by \cite[Theorem 1.2]{Wit24}. 
    Let $g \colon Y \to C$ be the proper contraction associated to $K_Y$. Note that $g_*\mathcal{O}_Y=\mathcal{O}_C$, which implies for the structure morphism $\pi_C \colon C \to \Spec(\mathbb{Z})$ that $(\pi_C)_* \mathcal{O}_C=\mathcal{O}_{\Spec(\mathbb{Z})}$ and thus $h^0(C_0, \mathcal{O}_{C_0})=1$. 
    As $H^1(C_0, \mathcal{O}_{C_0}) \hookrightarrow H^1(Y_0, \mathcal{O}_{Y_0})$ and $H^1(Y_0, \mathcal{O}_{Y_0})=0$ by \autoref{thm: abrashkin}, we deduce that $C_0$ is a curve of genus $0$. Note that we do not claim that $g_p \colon Y_p \to C_p$ is the semiample contraction associated to $K_{Y_p}$, but only the weaker statement that $C \simeq \mathbb{P}^1_{\mathbb{Z}}$. For this, we show that every closed fibre $C_p$ is a smooth conic.
    By construction, $C$ is normal and $\dim(C) = 2$. Note that the fibre $C_p$ is geometrically irreducible as $Y_p$ is.  
    Also, $C_p$ is reduced: first, it satisfies Serre's condition $S_1$ as it is a Cartier divisor on a normal surface. Therefore, if $C_p$ is not reduced, we have an equality of divisor $C_p=g^*(p)=nF$ for some integral $F$ and $n>1$. Then $Y_p=\pi^*(C_p)=n\pi^*F$ is also not reduced, contradicting the smoothness of $Y \to \mathbb{Z}$. 
    As $C_p$ is geometrically integral, $h^0(C_p, \mathcal{O}_{C_p})=1$ and by constancy of the Euler characteristic, we deduce that $h^1(C_p, \mathcal{O}_{C_p})=0$.    
    As $C_p$ is integral, its normalisation is an isomorphism and thus $C_p$ is a smooth conic. 
    We conclude the isomorphism $C \cong \mathbb{P}^1_\mathbb{Z}$ by \autoref{prop:curve_Z}.

Finally, for the case of minimal surfaces of general type we can combine the Noether formula with \autoref{thm: abrashkin} to deduce
    $10-b_2(Y_0)=K^2_{Y_0}$ and $b_2(Y_0)=\rho(Y_0)$.
    As $Y_p$ is a smooth surface and it lifts to $p^2$ the Kawamata--Viehweg vanishing holds by \cite[Corollaire 2.8.(ii)]{Deligne--Illusie} and thus $h^i(Y_p, \mathcal{O}_{Y_p}(mK_{Y_p}))=0$ for $i>0$ and $m \geq 2$. 
    The natural morphism $(Y \times \mathbb{F}_p)_{\can} \to Y_{\can} \times \mathbb{F}_p$ is an isomorphism thanks to \cite[Theorem 1.1]{EH21}.
    For every point $t \in \mathbb{Z}$, the linear system $|5K_{Y_{\can,t}}|$ is very ample on the fibre $Y_{\can,t}$ by \cite[Main theorem]{Eke88} and \cite{Bom73}.
    As $H^1(Y_{\can,t}, \mathcal{O}_{Y_{\can,t}}(5K_{Y_{\can,t}}))$ embeds in $H^1(Y_t, \mathcal{O}_{Y_t}(5K_{Y_t}))=0$, Grauert's theorem on higher direct images shows that $|5K_{Y_{\can}}|$ is very ample over $\mathbb{Z}$.
\end{proof}

\begin{remark}
    In case (2), we do not know whether the base change $Y_p \rightarrow \mathbb{P}^1_{\mathbb{F}_p}$ coincides with the semiample fibration associated to $K_{Y_p}$ (cf. the problem of asymptotic invariance of plurigenera for families of elliptic surfaces in \cite{EH21, Bri23_invariance_failure, BBS24}).
\end{remark}

\begin{corollary}
Let $X \to \Spec(\mathbb{Z)}$ be a smooth projective morphism of relative dimension 2.      
If $P_2(X_0)=0$, then $X_{0}$ is rational.
\end{corollary}

\begin{proof}
    By \autoref{thm: abrashkin}, we have $h^1(X_0, \mathcal{O}_{X_0})=0$.
    By Castelnuovo's theorem, $X_{\overline{\mathbb{Q}}}$ is a rational surface and thus $K_X$ is not pseudoeffective.
    Thus by \autoref{thm: rough_classification_Z}, we conclude that $X_0$ is rational as well.
\end{proof}


\subsubsection*{del Pezzo surfaces over $\mathbb{Z}$}
As an application, we recover a result of Scholl \cite{Sch85} on the classification of smooth del Pezzo surfaces over $\mathbb{Z}$.
We say that a collection $Z_i \subset \mathbb{P}^2_{\mathbb{Z}}$ of $\mathbb{Z}$-points are in general position if for every prime $p$ the reduction $Z_{i,p}$ are in general position in the classical sense \cite[Proposition 8.1.25]{Dol12}.

\begin{corollary} \label{cor: dP}
    Let $X \to \Spec(\mathbb{Z})$ be a smooth projective morphism  of relative dimension 2 such that $-K_X$ is ample.
    Then $K_X^2 \geq 5$ and $X$ is isomorphic to one of the following: 
    \begin{enumerate}
        \item $\mathbb{P}^1_{\mathbb{Z}} \times \mathbb{P}^1_{\mathbb{Z}}$;
        \item the blow-up of $\mathbb{P}^2_{\mathbb{Z}}$ along a subset of $\left\{[1:0:0], [0:1:0], [0:0:1], [1:1:1] \right\}$.
        \end{enumerate}

\end{corollary}

\begin{proof}
    The Mori fibre spaces with anticanonical ample class are $\mathbb{P}^2_\mathbb{Z}$, $\mathbb{P}^1_{\mathbb{Z}} \times \mathbb{P}^1_{\mathbb{Z}}$ and $\mathbb{P}_{\mathbb{P}^1_\mathbb{Z}}(\mathcal{O}_{\mathbb{P}^1_{\mathbb{Z}}} \oplus \mathcal{O}_{\mathbb{P}^1_{\mathbb{Z}}}(-1))$. 
    In the case of $\mathbb{P}(\mathcal{E}) \to \mathbb{P}^1_{\mathbb{Z}}$, the Fano condition implies for each point $(p) \in \Spec(\mathbb{Z})$, the fibre $X_p$ is either $\mathbb{P}^1 \times \mathbb{P}^1$ or $\mathbb{F}_1$. Note that, in this case, the type of every fibre depends exclusively on the type of the generic fibre by \autoref{prop:types_of_Hirzebruch_surfaces}.
    We can thus apply \autoref{lem: constancy_Hirzebruch} to conclude.

    As in the classical case, we reduce to study which $\mathbb{Z}$-points $Z_i$ on $\mathbb{P}^2_{\mathbb{Z}}$ we are allowed to blow-up.
    By hypothesis, $X_p$ is a del Pezzo surface for every prime $p$, which is equivalent to asking $Z_{i}$ to be in general position. 
    In particular, no more than two points must lie on the same line after reduction modulo $p$.
    As $\mathbb{Z}$ is a PID, $\mathbb{P}^2(\mathbb{Z})$ coincides with $(\mathbb{Z}^{3} \setminus \left\{(0,0,0)\right\}) / {\mathbb{Z}}^{*}$.
    Let $\left\{Z_i=[a_i:b_i:c_i]\right\}_{i=1}^3$ be an ordered set of three $\mathbb{Z}$-points in general position and let $A \in M_{3 \times 3}(\mathbb{Z)}$ be the matrix for which $Ae_i=Z_i$. As the $Z_i$ are not collinear modulo every $p$, we deduce that $\det(A) \neq 0$ modulo every prime $p$. This implies that $\det(A) \in \mathbb{Z}^{*}$ and thus $A \in \GL_3(\mathbb{Z})$.
    We may thus assume that $Z_1=[1:0:0], Z_2=[0:1:0]$ and $Z_3=[0:0:1]$.
    Let $Z_4=[a:b:c]$ be a fourth $\mathbb{Z}$-point in general position with respect to $Z_i$, equivalently $Z_{4,p}$ does not intersect the union of three lines $(xyz=0)$ modulo $p$.
    This implies that $Z_4$ belongs to the standard torus $\mathbb{G}_m^2(\mathbb{Z})=\left\{[x_0:x_1:x_2] | x_i \in \mathbb{Z^{*}} \right\}$. Acting by rescaling, we see that $Z_4$ is projectively equivalent to $[1:1:1]$. From now on, we assume $Z_4=[1:1:1]$.
    
    Suppose by contradiction that there exists a fifth $\mathbb{Z}$-point $Z_5$ distinct from $Z_1, \dots, Z_4$ in general position. 
    As $Z_5$ is in general position with respect to $Z_1, \dots, Z_3$, $Z_{5} \in \mathbb{G}_m^2(\mathbb{Z)}$ and thus its reduction modulo 2 must be equal to $Z_{4,p=2}$, thus reaching a contradiction.
\end{proof}

\begin{remark}
  The analogous conclusion of \autoref{cor: dP} also holds if we replace $\mathbb{Z}$ with a ring of integers $\mathcal{O}_K$ satisfying \ref{item:Minkowski}, \Bref{2}, \PIDref{1} and that admits a closed point with residue field $\mathbb{F}_2$ (for example, $\mathcal{O}_K=\mathbb{Z}[i]$).
\end{remark}

\end{example}

    \bibliographystyle{amsalpha}
	\bibliography{refs}

\begin{bibdiv}
\begin{biblist}

\bib{Abr90}{article}{
      author={Abrashkin, V.~A.},
       title={Modular representations of the {G}alois group of a local field
  and a generalization of a conjecture of {S}hafarevich},
        date={1989},
        ISSN={0373-2436},
     journal={Izv. Akad. Nauk SSSR Ser. Mat.},
      volume={53},
      number={6},
       pages={1135\ndash 1182, 1337},
         url={https://doi.org/10.1070/IM1990v035n03ABEH000715},
      review={\MR{1039960}},
}

\bib{Artin}{incollection}{
      author={Artin, M.},
       title={Algebraization of formal moduli. {I}},
        date={1969},
   booktitle={Global {A}nalysis ({P}apers in {H}onor of {K}. {K}odaira)},
   publisher={Univ. Tokyo Press, Tokyo},
       pages={21\ndash 71},
      review={\MR{260746}},
}

\bib{Bad01}{book}{
      author={B\u{a}descu, Lucian},
       title={Algebraic surfaces},
      series={Universitext},
   publisher={Springer-Verlag, New York},
        date={2001},
        ISBN={0-387-98668-5},
         url={https://doi.org/10.1007/978-1-4757-3512-3},
        note={Translated from the 1981 Romanian original by Vladimir Ma\c{s}ek
  and revised by the author},
      review={\MR{1805816}},
}

\bib{BBS24}{article}{
      author={Bernasconi, Fabio},
      author={Brivio, Iacopo},
      author={Stigant, Liam},
       title={Abundance theorem for threefolds in mixed characteristic},
        date={2024},
        ISSN={0025-5831},
     journal={Math. Ann.},
      volume={388},
      number={1},
       pages={141\ndash 166},
         url={https://doi.org/10.1007/s00208-022-02514-5},
      review={\MR{4693931}},
}

\bib{Ber21}{article}{
      author={Bernasconi, Fabio},
       title={On the base point free theorem for klt threefolds in large
  characteristic},
        date={2021},
        ISSN={0391-173X},
     journal={Ann. Sc. Norm. Super. Pisa Cl. Sci. (5)},
      volume={22},
      number={2},
       pages={583\ndash 600},
      review={\MR{4288666}},
}

\bib{NeronModels}{book}{
      author={Bosch, Siegfried},
      author={L\"utkebohmert, Werner},
      author={Raynaud, Michel},
       title={N\'eron models},
      series={Ergebnisse der Mathematik und ihrer Grenzgebiete (3) [Results in
  Mathematics and Related Areas (3)]},
   publisher={Springer-Verlag, Berlin},
        date={1990},
      volume={21},
        ISBN={3-540-50587-3},
         url={https://doi.org/10.1007/978-3-642-51438-8},
      review={\MR{1045822}},
}

\bib{7authors}{article}{
      author={Bhatt, Bhargav},
      author={Ma, Linquan},
      author={Patakfalvi, {Zs}olt},
      author={Schwede, Karl},
      author={Tucker, Kevin},
      author={Waldron, Joe},
      author={Witaszek, Jakub},
       title={Globally +-regular varieties and the minimal model program for
  threefolds in mixed characteristic},
        date={2023},
        ISSN={0073-8301},
     journal={Publ. Math. Inst. Hautes \'{E}tudes Sci.},
      volume={138},
       pages={69\ndash 227},
         url={https://doi.org/10.1007/s10240-023-00140-8},
      review={\MR{4666931}},
}

\bib{Bom73}{article}{
      author={Bombieri, E.},
       title={Canonical models of surfaces of general type},
        date={1973},
        ISSN={0073-8301,1618-1913},
     journal={Inst. Hautes \'{E}tudes Sci. Publ. Math.},
      number={42},
       pages={171\ndash 219},
         url={http://www.numdam.org/item?id=PMIHES_1973__42__171_0},
      review={\MR{318163}},
}

\bib{Bri23_invariance_failure}{article}{
      author={Brivio, Iacopo},
       title={Invariance of plurigenera fails in positive and mixed
  characteristic},
        date={2023},
        ISSN={1073-7928},
     journal={Int. Math. Res. Not. IMRN},
      number={24},
       pages={21215\ndash 21225},
         url={https://doi.org/10.1093/imrn/rnac304},
      review={\MR{4682759}},
}

\bib{Brivio25}{article}{
      author={Brivio, Iacopo},
       title={Non-extendable {MMP}s},
        date={2025},
     journal={arXiv e-print},
        note={Available at
  \href{https://arxiv.org/abs/2506.23883}{arXiv:2506.23883}},
}

\bib{CJM22}{incollection}{
      author={Cheng, Raymond},
      author={Ji, Lena},
      author={Larson, Matt},
      author={Olander, Noah},
       title={Theorem of the base},
        date={2022},
   booktitle={Stacks {P}roject {E}xpository {C}ollection ({SPEC})},
      series={London Math. Soc. Lecture Note Ser.},
      volume={480},
   publisher={Cambridge Univ. Press, Cambridge},
       pages={163\ndash 193},
      review={\MR{4480536}},
}

\bib{CTsemiampleness}{article}{
      author={Cascini, Paolo},
      author={Tanaka, Hiromu},
       title={Relative semi-ampleness in positive characteristic},
        date={2020},
        ISSN={0024-6115,1460-244X},
     journal={Proc. Lond. Math. Soc. (3)},
      volume={121},
      number={3},
       pages={617\ndash 655},
         url={https://doi.org/10.1112/plms.12323},
      review={\MR{4100119}},
}

\bib{brauergroupbook}{book}{
      author={Colliot-Th\'{e}l\`ene, Jean-Louis},
      author={Skorobogatov, Alexei~N.},
       title={The {B}rauer-{G}rothendieck group},
      series={Ergebnisse der Mathematik und ihrer Grenzgebiete. 3. Folge. A
  Series of Modern Surveys in Mathematics [Results in Mathematics and Related
  Areas. 3rd Series. A Series of Modern Surveys in Mathematics]},
   publisher={Springer, Cham},
        date={2021},
      volume={71},
        ISBN={978-3-030-74247-8; 978-3-030-74248-5},
         url={https://doi.org/10.1007/978-3-030-74248-5},
      review={\MR{4304038}},
}

\bib{Deligne--Illusie}{article}{
      author={Deligne, Pierre},
      author={Illusie, Luc},
       title={Rel\`evements modulo {$p^2$} et d\'{e}composition du complexe de
  de {R}ham},
        date={1987},
        ISSN={0020-9910},
     journal={Invent. Math.},
      volume={89},
      number={2},
       pages={247\ndash 270},
         url={https://doi.org/10.1007/BF01389078},
      review={\MR{894379}},
}

\bib{Dol12}{book}{
      author={Dolgachev, Igor~V.},
       title={Classical algebraic geometry},
   publisher={Cambridge University Press, Cambridge},
        date={2012},
        ISBN={978-1-107-01765-8},
         url={https://doi.org/10.1017/CBO9781139084437},
        note={A modern view},
      review={\MR{2964027}},
}

\bib{EH21}{article}{
      author={Egbert, Andrew},
      author={Hacon, Christopher~D.},
       title={Invariance of certain plurigenera for surfaces in mixed
  characteristics},
        date={2021},
        ISSN={0027-7630},
     journal={Nagoya Math. J.},
      volume={243},
       pages={1\ndash 10},
         url={https://doi.org/10.1017/nmj.2019.28},
      review={\MR{4298650}},
}

\bib{Eke88}{article}{
      author={Ekedahl, Torsten},
       title={Canonical models of surfaces of general type in positive
  characteristic},
        date={1988},
        ISSN={0073-8301},
     journal={Inst. Hautes \'{E}tudes Sci. Publ. Math.},
      number={67},
       pages={97\ndash 144},
         url={http://www.numdam.org/item?id=PMIHES_1988__67__97_0},
      review={\MR{972344}},
}

\bib{Fon85}{article}{
      author={Fontaine, Jean-Marc},
       title={Il n'y a pas de vari\'{e}t\'{e} ab\'{e}lienne sur {${\bf Z}$}},
        date={1985},
        ISSN={0020-9910},
     journal={Invent. Math.},
      volume={81},
      number={3},
       pages={515\ndash 538},
         url={https://doi.org/10.1007/BF01388584},
      review={\MR{807070}},
}

\bib{Fontaine_schemas_Z}{inproceedings}{
      author={Fontaine, Jean-Marc},
       title={Sch\'{e}mas propres et lisses sur {${\bf Z}$}},
        date={1993},
   booktitle={Proceedings of the {I}ndo-{F}rench {C}onference on {G}eometry
  ({B}ombay, 1989)},
   publisher={Hindustan Book Agency, Delhi},
       pages={43\ndash 56},
      review={\MR{1274493}},
}

\bib{EGA_III_2}{article}{
      author={Grothendieck, A.},
       title={\'{E}l\'{e}ments de g\'{e}om\'{e}trie alg\'{e}brique. {III}.
  \'{E}tude cohomologique des faisceaux coh\'{e}rents. {II}},
        date={1963},
        ISSN={0073-8301,1618-1913},
     journal={Inst. Hautes \'{E}tudes Sci. Publ. Math.},
      number={17},
       pages={91},
         url={http://www.numdam.org/item?id=PMIHES_1963__17__91_0},
      review={\MR{163911}},
}

\bib{Han78}{article}{
      author={Hanna, Charles~C.},
       title={Subbundles of vector bundles on the projective line},
        date={1978},
        ISSN={0021-8693},
     journal={J. Algebra},
      volume={52},
      number={2},
       pages={322\ndash 327},
         url={https://doi.org/10.1016/0021-8693(78)90242-9},
      review={\MR{480506}},
}

\bib{Har77}{book}{
      author={Hartshorne, Robin},
       title={Algebraic geometry},
      series={Graduate Texts in Mathematics, No. 52},
   publisher={Springer-Verlag, New York-Heidelberg},
        date={1977},
        ISBN={0-387-90244-9},
      review={\MR{463157}},
}

\bib{Hartshorne_Stable_reflexive_sheaves}{article}{
      author={Hartshorne, Robin},
       title={Stable reflexive sheaves},
        date={1980},
        ISSN={0025-5831},
     journal={Math. Ann.},
      volume={254},
      number={2},
       pages={121\ndash 176},
      review={\MR{MR597077 (82b:14011)}},
}

\bib{HW22}{article}{
      author={Hacon, Christopher~D.},
      author={Witaszek, Jakub},
       title={On the relative minimal model program for threefolds in low
  characteristics},
        date={2022},
        ISSN={2096-6075,2524-7182},
     journal={Peking Math. J.},
      volume={5},
      number={2},
       pages={365\ndash 382},
         url={https://doi.org/10.1007/s42543-021-00037-7},
      review={\MR{4492657}},
}

\bib{HW23}{article}{
      author={Hacon, Christopher~D.},
      author={Witaszek, Jakub},
       title={On the relative minimal model program for fourfolds in positive
  and mixed characteristic},
        date={2023},
        ISSN={2050-5086},
     journal={Forum Math. Pi},
      volume={11},
       pages={Paper No. e10, 35},
         url={https://doi.org/10.1017/fmp.2023.6},
      review={\MR{4565409}},
}

\bib{Iakovenko}{article}{
      author={Iakovenko, S.~S.},
       title={Vector bundles on {{\(\mathbb{P}^1_{\mathbb{Z}}\)}} with generic
  fiber {{\(\mathcal{O}\otimes\mathcal{O}(1)\)}} and simple jumps},
    language={English},
        date={2018},
        ISSN={1072-3374},
     journal={J. Math. Sci., New York},
      volume={232},
      number={5},
       pages={732\ndash 745},
}

\bib{ito2024arithmeticfinitenessmukaivarieties}{article}{
      author={Ito, Tetsushi},
      author={Kanemitsu, Akihiro},
      author={Takamatsu, Teppei},
      author={Tanaka, Yuuji},
       title={Arithmetic finiteness of {M}ukai varieties of genus 7},
        date={2024},
     journal={arXiv e-print},
        note={Available at
  \href{https://arxiv.org/abs/2409.20046}{arXiv:2409.20046}},
}

\bib{ito2025quinticdelpezzothreefolds}{article}{
      author={Ito, Tetsushi},
      author={Kanemitsu, Akihiro},
      author={Takamatsu, Teppei},
      author={Tanaka, Yuuji},
       title={Quintic del {P}ezzo threefolds in positive and mixed
  characteristic},
        date={2025},
     journal={arXiv e-print},
        note={Available at
  \href{https://arxiv.org/abs/2506.15086}{arXiv:2506.15086}},
}

\bib{itov22}{article}{
      author={Ito, Tetsushi},
      author={Kanemitsu, Akihiro},
      author={Takamatsu, Teppei},
      author={Tanaka, Yuuji},
       title={Fano threefolds of genus 12 with large automorphism group in
  positive and mixed characteristic},
        date={2026},
     journal={arXiv e-print},
        note={Available at
  \href{https://arxiv.org/abs/2601.10106}{arXiv:2601.10106}},
}

\bib{Kaw94}{article}{
      author={Kawamata, Yujiro},
       title={Semistable minimal models of threefolds in positive or mixed
  characteristic},
        date={1994},
        ISSN={1056-3911,1534-7486},
     journal={J. Algebraic Geom.},
      volume={3},
      number={3},
       pages={463\ndash 491},
      review={\MR{1269717}},
}

\bib{kk-singbook}{book}{
      author={Koll\'{a}r, J\'{a}nos},
       title={Singularities of the minimal model program},
      series={Cambridge Tracts in Mathematics},
   publisher={Cambridge University Press, Cambridge},
        date={2013},
      volume={200},
        ISBN={978-1-107-03534-8},
         url={https://doi.org/10.1017/CBO9781139547895},
        note={With a collaboration of S\'{a}ndor Kov\'{a}cs},
      review={\MR{3057950}},
}

\bib{KU-85}{article}{
      author={Katsura, Toshiyuki},
      author={Ueno, Kenji},
       title={On elliptic surfaces in characteristic {$p$}},
        date={1985},
        ISSN={0025-5831},
     journal={Math. Ann.},
      volume={272},
      number={3},
       pages={291\ndash 330},
         url={https://doi.org/10.1007/BF01455561},
      review={\MR{799664}},
}

\bib{Laz1}{book}{
      author={Lazarsfeld, Robert},
       title={Positivity in algebraic geometry. {I}},
      series={Ergebnisse der Mathematik und ihrer Grenzgebiete. 3. Folge. A
  Series of Modern Surveys in Mathematics [Results in Mathematics and Related
  Areas. 3rd Series. A Series of Modern Surveys in Mathematics]},
   publisher={Springer-Verlag, Berlin},
        date={2004},
      volume={48},
        ISBN={3-540-22533-1},
         url={https://doi.org/10.1007/978-3-642-18808-4},
        note={Classical setting: line bundles and linear series},
      review={\MR{2095471}},
}

\bib{Liu}{book}{
      author={Liu, Qing},
       title={Algebraic geometry and arithmetic curves},
      series={Oxford Graduate Texts in Mathematics},
   publisher={Oxford University Press, Oxford},
        date={2002},
      volume={6},
        ISBN={0-19-850284-2},
        note={Translated from the French by Reinie Ern\'{e}, Oxford Science
  Publications},
      review={\MR{1917232}},
}

\bib{lmfdb}{misc}{
      author={{LMFDB Collaboration}, The},
       title={The {L}-functions and modular forms database},
         how={\url{https://www.lmfdb.org}},
        date={2026},
        note={[Online; accessed 11 July 2026]},
}

\bib{Maz86}{article}{
      author={Mazur, Barry},
       title={Arithmetic on curves},
        date={1986},
        ISSN={0273-0979,1088-9485},
     journal={Bull. Amer. Math. Soc. (N.S.)},
      volume={14},
      number={2},
       pages={207\ndash 259},
         url={https://doi.org/10.1090/S0273-0979-1986-15430-3},
      review={\MR{828821}},
}

\bib{MR81}{incollection}{
      author={Moret-Bailly, Laurent},
       title={Familles de courbes et de vari\'{e}t\'{e}s ab\'{e}liennes sur
  {${\Bbb P}^1$}. {II}. {E}xemples},
        date={1981},
       pages={125\ndash 140},
        note={Seminar on Pencils of Curves of Genus at Least Two},
      review={\MR{3618576}},
}

\bib{Neu99}{book}{
      author={Neukirch, J\"{u}rgen},
       title={Algebraic number theory},
      series={Grundlehren der mathematischen Wissenschaften [Fundamental
  Principles of Mathematical Sciences]},
   publisher={Springer-Verlag, Berlin},
        date={1999},
      volume={322},
        ISBN={3-540-65399-6},
         url={https://doi.org/10.1007/978-3-662-03983-0},
        note={Translated from the 1992 German original and with a note by
  Norbert Schappacher, With a foreword by G. Harder},
      review={\MR{1697859}},
}

\bib{Polyakov}{article}{
      author={Polyakov, V.~M.},
       title={Finiteness of the number of classes of vector bundles on {$\Bbb
  P^1_{\Bbb Z}$} with jumps of height 2},
        date={2022},
        ISSN={0373-2703},
     journal={Zap. Nauchn. Sem. S.-Peterburg. Otdel. Mat. Inst. Steklov.
  (POMI)},
      volume={511},
       pages={137\ndash 160},
      review={\MR{4533327}},
}

\bib{quillen}{article}{
      author={Quillen, Daniel},
       title={Projective modules over polynomial rings},
    language={English},
        date={1976},
        ISSN={0020-9910},
     journal={Invent. Math.},
      volume={36},
       pages={167\ndash 171},
         url={https://eudml.org/doc/142417},
}

\bib{Raynaud}{article}{
      author={Raynaud, M.},
       title={Sp\'ecialisation du foncteur de {P}icard},
        date={1970},
        ISSN={0073-8301,1618-1913},
     journal={Inst. Hautes \'Etudes Sci. Publ. Math.},
      number={38},
       pages={27\ndash 76},
         url={http://www.numdam.org/item?id=PMIHES_1970__38__27_0},
      review={\MR{282993}},
}

\bib{Ros-Schroer}{article}{
      author={R\"{o}ssler, Damian},
      author={Schr\"{o}er, Stefan},
       title={Moret-{B}ailly families and non-liftable schemes},
        date={2022},
        ISSN={2313-1691,2214-2584},
     journal={Algebr. Geom.},
      volume={9},
      number={1},
       pages={93\ndash 121},
         url={https://doi.org/10.14231/ag-2022-004},
      review={\MR{4371550}},
}

\bib{Saf63}{inproceedings}{
      author={\v{S}afarevi\v{c}, I.~R.},
       title={Algebraic number fields},
        date={1963},
   booktitle={Proc. {I}nternat. {C}ongr. {M}athematicians ({S}tockholm, 1962)},
   publisher={Inst. Mittag-Leffler, Djursholm},
       pages={163\ndash 176},
      review={\MR{202709}},
}

\bib{Sch03}{article}{
      author={Schoof, Ren\'{e}},
       title={Abelian varieties over cyclotomic fields with good reduction
  everywhere},
        date={2003},
        ISSN={0025-5831,1432-1807},
     journal={Math. Ann.},
      volume={325},
      number={3},
       pages={413\ndash 448},
         url={https://doi.org/10.1007/s00208-002-0368-7},
      review={\MR{1968602}},
}

\bib{Schütt}{article}{
      author={Sch{\"u}tt, Matthias},
       title={{{\(\mathbb{Q}_\ell\)}}-cohomology projective planes and
  {Enriques} surfaces in characteristic two},
    language={English},
        date={2019},
        ISSN={2491-6765},
     journal={{\'E}pijournal de G{\'e}om. Alg{\'e}br., EPIGA},
      volume={3},
       pages={24},
         url={epiga.episciences.org/5602},
        note={Id/No 10},
}

\bib{Sch23}{article}{
      author={Schr\"{o}er, Stefan},
       title={There is no {E}nriques surface over the integers},
        date={2023},
        ISSN={0003-486X},
     journal={Ann. of Math. (2)},
      volume={197},
      number={1},
       pages={1\ndash 63},
         url={https://doi.org/10.4007/annals.2023.197.1.1},
      review={\MR{4513142}},
}

\bib{Sch25}{article}{
      author={Schoof, Ren{é}},
       title={Abelian varieties over real quadratic fields with good reduction
  everywhere},
        date={2025},
     journal={preprint},
        note={Available at
  \href{https://www.mat.uniroma2.it/~schoof/abrealAGTC.pdf}{https://www.mat.uniroma2.it/~schoof/abrealAGTC.pdf}},
}

\bib{schröer2025k3surfacessmallnumber}{article}{
      author={Schr\"{o}er, Stefan},
       title={K3 surfaces over small number fields and {K}ummer constructions
  in families},
        date={2025},
     journal={arXiv e-print},
        note={Available at
  \href{https://arxiv.org/abs/2506.14732}{arXiv:2506.14732}},
}

\bib{Sch85}{article}{
      author={Scholl, A.~J.},
       title={A finiteness theorem for del {P}ezzo surfaces over algebraic
  number fields},
        date={1985},
        ISSN={0024-6107},
     journal={J. London Math. Soc. (2)},
      volume={32},
      number={1},
       pages={31\ndash 40},
         url={https://doi.org/10.1112/jlms/s2-32.1.31},
      review={\MR{813382}},
}

\bib{Serre1972}{article}{
      author={Serre, Jean-Pierre},
       title={Propriétés galoisiennes des points d'ordre fini des courbes
  elliptiques.},
    language={fre},
        date={1971/72},
     journal={Inventiones mathematicae},
      volume={15},
       pages={259\ndash 331},
         url={http://eudml.org/doc/142133},
}

\bib{Set81}{article}{
      author={Setzer, Bennett},
       title={Elliptic curves with good reduction everywhere over quadratic
  fields and having rational {$j$}-invariant},
        date={1981},
        ISSN={0019-2082},
     journal={Illinois J. Math.},
      volume={25},
      number={2},
       pages={233\ndash 245},
      review={\MR{607025}},
}

\bib{SGA6}{book}{
      author={SGA6},
       title={Th\'{e}orie des intersections et th\'{e}or\`eme de
  {R}iemann-{R}och},
      series={Lecture Notes in Mathematics, Vol. 225},
   publisher={Springer-Verlag, Berlin-New York},
        date={1971},
        note={S\'{e}minaire de G\'{e}om\'{e}trie Alg\'{e}brique du Bois-Marie
  1966--1967 (SGA 6), Dirig\'{e} par P. Berthelot, A. Grothendieck et L.
  Illusie. Avec la collaboration de D. Ferrand, J. P. Jouanolou, O. Jussila, S.
  Kleiman, M. Raynaud et J. P. Serre},
      review={\MR{354655}},
}

\bib{Smirnov15}{incollection}{
      author={Smirnov, A.},
       title={On filtrations of vector bundles over {$\Bbb P^1_{\Bbb Z}$}},
        date={2015},
   booktitle={Arithmetic and geometry},
      series={London Math. Soc. Lecture Note Ser.},
      volume={420},
   publisher={Cambridge Univ. Press, Cambridge},
       pages={436\ndash 457},
      review={\MR{3467134}},
}

\bib{Smi16}{article}{
      author={Smirnov, A.~L.},
       title={Vector bundles on {${\bf P}^1_{\Bbb Z}$} with simple jumps},
        date={2016},
        ISSN={0373-2703},
     journal={Zap. Nauchn. Sem. S.-Peterburg. Otdel. Mat. Inst. Steklov.
  (POMI)},
      volume={452},
      number={Voprosy Teorii Predstavleni\u{\i} Algebr i Grupp. 30},
       pages={202\ndash 217},
         url={https://doi.org/10.1007/s10958-018-3901-2},
      review={\MR{3589291}},
}

\bib{stacks-project}{misc}{
      author={{Stacks Project Authors}, The},
       title={\textit{Stacks Project}},
         how={\url{https://stacks.math.columbia.edu}},
        date={2018},
}

\bib{Sti24}{article}{
      author={Stigant, Liam},
       title={Termination of threefold flips in mixed characteristic},
        date={2024},
        ISSN={0024-6093,1469-2120},
     journal={Bull. Lond. Math. Soc.},
      volume={56},
      number={1},
       pages={72\ndash 78},
         url={https://doi.org/10.1112/blms.12914},
      review={\MR{4700579}},
}

\bib{Suslin}{article}{
      author={Suslin, A.~A.},
       title={Projective modules over a polynomial ring are free},
    language={English},
        date={1976},
        ISSN={0197-6788},
     journal={Sov. Math., Dokl.},
      volume={17},
       pages={1160\ndash 1164},
}

\bib{Tan18}{article}{
      author={Tanaka, Hiromu},
       title={Minimal model program for excellent surfaces},
        date={2018},
        ISSN={0373-0956},
     journal={Ann. Inst. Fourier (Grenoble)},
      volume={68},
      number={1},
       pages={345\ndash 376},
         url={https://doi.org/10.5802/aif.3163},
      review={\MR{3795482}},
}

\bib{Tan20}{article}{
      author={Tanaka, Hiromu},
       title={Pathologies on {M}ori fibre spaces in positive characteristic},
        date={2020},
        ISSN={0391-173X},
     journal={Ann. Sc. Norm. Super. Pisa Cl. Sci. (5)},
      volume={20},
      number={3},
       pages={1113\ndash 1134},
      review={\MR{4166801}},
}

\bib{TY23}{article}{
      author={Takamatsu, Teppei},
      author={Yoshikawa, Shou},
       title={Minimal model program for semi-stable threefolds in mixed
  characteristic},
        date={2023},
        ISSN={1056-3911},
     journal={J. Algebraic Geom.},
      volume={32},
      number={3},
       pages={429\ndash 476},
      review={\MR{4622257}},
}

\bib{Wit24}{article}{
      author={Witaszek, Jakub},
       title={Relative semiampleness in mixed characteristic},
        date={2024},
        ISSN={0012-7094},
     journal={Duke Math. J.},
      volume={173},
      number={9},
       pages={1631\ndash 1675},
         url={https://doi.org/10.1215/00127094-2023-0053},
      review={\MR{4766840}},
}

\bib{XX24}{article}{
      author={Xie, Lingyao},
      author={Xue, Qingyuan},
       title={On the termination of the {MMP} for semistable fourfolds in mixed
  characteristic},
        date={2024},
        ISSN={0026-2285,1945-2365},
     journal={Michigan Math. J.},
      volume={74},
      number={5},
       pages={977\ndash 999},
         url={https://doi.org/10.1307/mmj/20216172},
      review={\MR{4822672}},
}

\bib{YamamuraImaginary}{article}{
      author={Yamamura, Ken},
       title={Maximal unramified extensions of imaginary quadratic number
  fields of small conductors},
        date={1997},
        ISSN={1246-7405,2118-8572},
     journal={J. Th\'eor. Nombres Bordeaux},
      volume={9},
      number={2},
       pages={405\ndash 448},
         url={https://doi.org/10.5802/jtnb.211},
      review={\MR{1617407}},
}

\end{biblist}
\end{bibdiv}
	
\end{document}